\documentclass[12pt]{amsart}
\usepackage{amsmath,amsfonts,amssymb,amsthm,amstext,pgf,graphicx,hyperref,verbatim,lmodern,textcomp,color,young,tikz}
\usepackage{mleftright}
\usepackage{float}
\usetikzlibrary{decorations}
\usepackage[mathscr]{euscript}
\usetikzlibrary{decorations.markings}
\usetikzlibrary{arrows}
\theoremstyle{plain}
\newtheorem{theorem}{Theorem}[section]
\newtheorem{combinatorial proof}{Theorem}[combinatorial proof]
\newtheorem{lemma}[theorem]{Lemma}
\newtheorem{example}[theorem]{Example}
\newtheorem{proposition}[theorem]{Proposition}

\newtheorem{rem}{Remark}

\theoremstyle{definition}
\newtheorem{defn}{Definition}[section]

\title{}
\begin{document}
\title[enumeration of weighted paths on a DIGRAPH AND block HOOK determinant ]{enumeration of weighted paths on a DIGRAPH AND block HOOK determinant}
%\author[Sajal Kumar Mukherjee]{Sajal Kumar Mukherjee}
\author[Sudip Bera]{Sudip Bera}
\address[Sudip Bera]{Department of Mathematics, Indian Institute of Science, Bangalore 560 012}
\email{sudipbera@iisc.ac.in}
\maketitle
\begin{abstract}
In this article, we evaluate determinants of ``block hook'' matrices, which are block matrices consist of hook matrices. In particular, we deduce that the determinant of a block hook matrix factorizes nicely. In addition we give a combinatorial interpretation of the aforesaid factorization property by counting weighted paths in a suitable weighted digraph.
\end{abstract}  
\section{introduction}
The evaluation of determinants is a nice topic, and fascinating for many people \cite{27,26,25,39,36}. A huge amount of such evaluation has been collected in \cite{26,25}. Specially, the problem of calculating the determinant of a $2 \times 2$ block matrix has been long studied \cite{39,36}. Block matrices are applied all over in  mathematics and physics. They appear naturally in the description of systems with multiple discrete variables  \cite{29, 30}. Moreover, block matrices are utilized in many computational methods familiar to researchers of fluid dynamics \cite{34}. Also the determinants of these matrices are found over a large number of area for both analytical and numerical applications \cite{32,33}. The purpose of this paper is to evaluate the determinantal formulas of some special classes of block matrices, known as \emph{block hook matrix} (defined later). In particular, we will show that the determinants of these block hook matrices admit nice product formulas. Now, let us define hook matrix in a precise way. First we need to define the following.
\begin{defn}
A square matrix is called an \emph{hook} matrix if the pattern of the entries satisfy one of the following four conditions;
\begin{enumerate}
	\item[(a)] 
all the entries right and below of the entry at the  $(i,i)^{\text{th}}\;(i=1,\cdots, m)$ position are same
\item[(b)] 
all the entries left and above of the entry at the  $(m+1-i,m+1-i)^{\text{th}}\;(i=1,\cdots, m)$ position are same
\item[(c)] 
all the entries right and above of the entry at the  $(m+1-i,i)^{\text{th}}\;(i=1,\cdots m)$ position are same
\item[(d)] 
 all the entries left and below of the entry at the  $(i,m+1-i)^{\text{th}}\;(i=1,\cdots m)$ position are same.	
\end{enumerate}   	
\end{defn}
Now consider the matrices in \eqref{hookmat:2a}, \eqref{hookmat:2b},  \eqref{hookmat:8} and \eqref{hookmat:7}.
\begin{equation}\label{hookmat:2a}
A_m{(x_1, \cdots, x_m)}=\left(
\begin{array}{ccccc}
x_m & x_{m} & \cdots&x_m & x_m\\
x_m&x_{m-1} & \cdots&x_{m-1} & x_{m-1} \\
\vdots&\vdots&\ddots&\vdots&\vdots\\
x_{m} & x_{m-1}&\cdots & x_{2} & x_{2}  \\
x_m & x_{m-1}&\cdots &  x_{2} & x_1 
\end{array}
\right)
\end{equation}
\begin{equation}\label{hookmat:2b}
B_m{(x_1, \cdots, x_m)}=\left(
\begin{array}{ccccc}
x_1 & x_{2} & \cdots&x_{m-1} & x_m\\
x_2&x_{2} & \cdots&x_{m-1} & x_{m} \\
\vdots&\vdots&\ddots&\vdots&\vdots\\
x_{m-1} & x_{m-1}&\cdots & x_{m-1} & x_{m}  \\
x_m & x_{m}&\cdots &  x_{m} & x_m 
\end{array}
\right)
\end{equation}
\begin{align}\label{hookmat:8}
C_m(x_1, \cdots,  x_m)= \left(
\begin{array}{ccccc}
x_m & x_{m-1} &  \cdots &x_{2} & x_1\\
x_m&x_{m-1} &  \cdots & x_{2} & x_{2} \\
\vdots & \vdots & \ddots&\vdots &  \vdots  \\
x_{m} & x_{m-1} & \cdots &x_{m-1} & x_{m-1}  \\
x_m & x_{m} & \cdots & x_{m} & x_m \\ 
\end{array}
\right).
\end{align}
\begin{align}\label{hookmat:7}
D_m(x_1,  \cdots,  x_m)= \left(
\begin{array}{ccccc}
x_m & x_{m} &  \cdots &x_{m} & x_m\\
x_{m-1}&x_{m-1} &  \cdots & x_{m-1} & x_{m} \\
\vdots & \vdots & \ddots&\vdots &  \vdots  \\
x_{2} & x_{2} & \cdots &x_{m-1} & x_{m}  \\
x_1 & x_{2} & \cdots & x_{m-1} & x_m \\ 
\end{array}
\right).
\end{align}
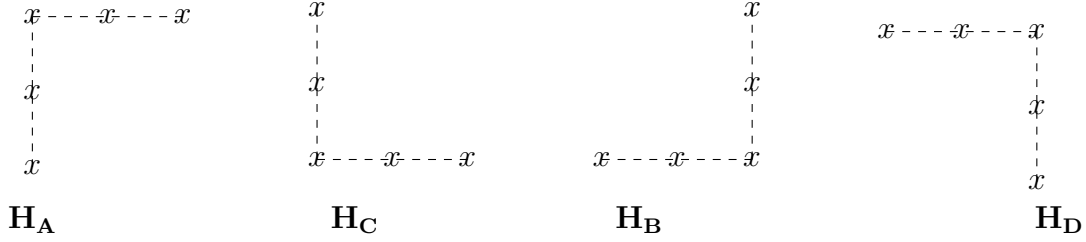
\begin{figure}   
	\begin{tikzpicture}
	\draw[dashed](0,0) -- (2,0);
	\draw [dashed](0,0) -- (0,-2);
	\node (B1) at (0,-2.7 )   {$\bf{H_A}$};
	\node (B1) at (0,0 )   {${x}$};
	\node (B1) at (1,0 )   {${x}$};
	\node (B1) at (2,0 )   {${x}$};
	\node (B1) at (0,-1 )   {${x}$};
	\node (B1) at (0,-2 )   {${x}$};
	\end{tikzpicture}
	\hspace{1cm}
	\begin{tikzpicture}
	\draw [dashed] (0,0) -- (2,0);
	\draw [dashed](0,0) -- (0,2);
	\node (B1) at (.5,-.8 )   {$\bf{H_C}$};
	\node (B1) at (0,0 )   {${x}$};
	\node (B1) at (1,0 )   {${x}$};
	\node (B1) at (2,0 )   {${x}$};
	\node (B1) at (0,1 )   {${x}$};
	\node (B1) at (0,2 )   {${x}$};
	\end{tikzpicture}
		\hspace{1cm}
\begin{tikzpicture}
	\draw [dashed] (0,0) -- (2,0);
	\draw [dashed](2,0) -- (2,2);
	\node (B1) at (.5,-.8 )   {$\bf{H_B}$};
	\node (B1) at (0,0 )   {${x}$};
	\node (B1) at (1,0 )   {${x}$};
	\node (B1) at (2,0 )   {${x}$};
	\node (B1) at (2,1 )   {${x}$};
	\node (B1) at (2,2 )   {${x}$};
	\end{tikzpicture}
	\hspace{1cm}
	\begin{tikzpicture}
	\draw [dashed](0,0) -- (2,0);
	\draw [dashed](2,0) -- (2,-2);
	\node (B1) at (0,0 )   {${x}$};
	\node (B1) at (1,0 )   {${x}$};
	\node (B1) at (2,0 )   {${x}$};
	\node (B1) at (2,-1 )   {${x}$};
	\node (B1) at (2,-2 )   {${x}$};
	\node (B1) at (2.3,-2.5 )   {$\bf{H_D}$};
	\end{tikzpicture}
	\caption{$H_A, H_C, H_B, H_D$ are four different shapes hook.}
	\label{hookfigexmple}
\end{figure} 
For each of the above matrix the pattern should be clear: In case of the matrix in \eqref{hookmat:2a}, all the entries right and below of the entry at the  $(i,i)^{\text{th}}\;(i=1,\cdots, m)$ position are same and clearly these entries form a hook of shape like $H_A$ in Figure \ref{hookfigexmple}. Similarly, for the matrix in \eqref{hookmat:7}, all the entries left and below of the entry at the  $(i,m+1-i)^{\text{th}}\;(i=1,\cdots m)$ position are same and these entries form a hook of shape like $H_D$ in Figure \ref{hookfigexmple}. Similarly the  entries in the matrices in \eqref{hookmat:2b} and \eqref{hookmat:8} form a hook of shape like $H_B$ and $H_C$ (in Figure \ref{hookfigexmple}) respectively. Moreover we have the following nice product formulas regarding the determinant of these hook matrices.
\begin{proposition}\label{elem hook them Am}
Let $A_m(x_1, \cdots, x_m)$ be a matrix defined as \eqref{hookmat:2a}. Then 
\[\det(A_m(x_1, \cdots, x_m))=  \prod_{i=1}^{m}(x_{i}-x_{i+1}), \text{ where } x_{m+1}=0.\]
\end{proposition}
\begin{proof}
We apply the following row operations on the matrix $A_m(x_1, \cdots, x_m);$
\[R'_i=R_i-R_{i-1}, \text{ for }i=2, \cdots, m .\] Clearly, the resulting matrix (matrix obtained after row operations) is a diagonal matrix with diagonal entries $(x_{i}-x_{i+1}),$ for $i=m, m-1, \cdots, 1,$ where $x_{m+1}=0.$ Hence the result.	
\end{proof}
\begin{proposition}\label{elem hook thm B_m}
Let $B_m(x_1, \cdots, x_m)$ be a matrix defined as \eqref{hookmat:2b}.	Then\[\det B_m(x_1, \cdots, x_m)=\prod_{i=1}^{m}(x_{i}-x_{i+1}), \text{ where } x_{m+1}=0.\]
\end{proposition}
\begin{proof}
We apply the following  row and column interchange on the matrix $A_m(x_1, \cdots, x_m);$ 
\begin{equation}\label{important row operation}
R_i\leftrightarrow R_{m-i+1}\text{ and } C_i\leftrightarrow C_{m-i+1}, \text{ for } i=
	\begin{cases}
1, 2, \cdots, \frac{m}{2}, & \text{if } m \text{ is even  } \\
1, 2, \cdots, \frac{m+1}{2}-1, & \text{ if } m \text{ is odd  },\
\end{cases}
\end{equation} 
Evidently the resulting matrix is $B_m(x_1, \cdots, x_m).$ This completes the proof. 
\end{proof}
\begin{proposition}\label{elem hook thm C_m}
Let $C_m(x_1, \cdots, x_m)$ be a matrix defined as \eqref{hookmat:8}. Then \[\det C_m(x_1, \cdots,  x_m)=
\begin{cases}
(-1)^{\frac{m}{2}}\times \prod\limits_{i=1}^{m}(x_{i}-x_{i+1}), & \text{if } m \text{ is even  } \\
(-1)^{\frac{m-1}{2}}\times \prod\limits_{i=1}^{m}(x_{i}-x_{i+1}), & \text{ if } m \text{ is odd  },\
\end{cases}\] \text{ where } $x_{m+1}=0.$
\end{proposition}
\begin{proof}
Here we apply the same row operations depicted as \eqref{important row operation} on the matrix $A_m(x_1, \cdots, x_m),$ and as a result we get the matrix $C_m(x_1, \cdots, x_m).$ So the proposition. 
\end{proof}
\begin{proposition}\label{elem hook thm D_m}
Let $D_m(x_1, \cdots, x_m)$ be a matrix defined as \eqref{hookmat:7}. Then \[\det D_m(x_1, \cdots, x_m)=
\begin{cases}
(-1)^{\frac{m}{2}}\times \prod\limits_{i=1}^{m}(x_{i}-x_{i+1}), & \text{if } m \text{ is even  } \\
(-1)^{\frac{m-1}{2}}\times \prod\limits_{i=1}^{m}(x_{i}-x_{i+1}), & \text{ if } m \text{ is odd,  }\
\end{cases}\] where $x_{m+1}=0.$
\end{proposition}
\begin{proof}
Applying the same column operations described as \eqref{important row operation} on the matrix $A_m(x_1, \cdots, x_m),$ we get the matrix $D_m(x_1, \cdots, x_m).$ Hence the proposition. 
\end{proof}
\begin{defn} 
An $N\times N (N\geq2)$ block matrix $M$ is called a \emph{block hook matrix} if each block of $M$ is a hook matrix. 
\end{defn}
\begin{rem}
If $N=1,$ then the block hook matrix reduces to a hook matrix. 
\end{rem} 
\begin{example}\label{gen hook exam}
\begin{equation*}
\renewcommand\arraystretch{1.3}
\mleft(
\begin{array}{cc|cc}
x_2 & x_2&y_2& y_2 \\
x_2&x_1&y_2&y_1 \\
\hline
z_2&z_2&w_2&w_2\\
z_2&z_1&w_2&w_1
\end{array}
\mright), \renewcommand\arraystretch{1.3}
\mleft(
\begin{array}{cc|cc}
x_1 & x_2&y_2& y_1 \\
x_2&x_2&y_2&y_2 \\
\hline
z_1&z_2&w_2&w_1\\
z_2&z_2&w_2&w_2
\end{array}
\mright), \renewcommand\arraystretch{1.3}
\mleft(
\begin{array}{cc|cc}
x_1 & x_2&y_2& y_1 \\
x_2&x_2&y_2&y_2 \\
\hline
z_2&z_2&w_2&w_2\\
z_1&z_2&w_2&w_1
\end{array}
\mright),
\renewcommand\arraystretch{1.3}
\mleft(
\begin{array}{cc|cc}
x_2 & x_2&y_2& y_2 \\
x_2&x_1&y_1&y_2 \\
\hline
z_2&z_2&w_2&w_2\\
z_2&z_1&w_1&w_2
\end{array}
\mright).
\end{equation*}
\end{example}
Example \ref{gen hook exam} contains four $2\times 2$ block matrices, in which each block is a hook. In fact, the $1^{st}$ matrix contains four hooks of shape $H_A$ as depicted in Figure \ref{hookfigexmple}, whereas the hooks in the last matrix are of different shapes. Also it can be shown that, the determinant of any one of the above matrices gives a nice product formula. So, one natural question occurs: Is it true that the determinant of any block hook matrix admits such a product formula? In this paper, we formulate different classes of block hook matrices of order $Nm$ ($N\times N$ block matrix and order of each block is $m$ ) and prove that the determinants of all these block hook matrices can be written as $\prod\limits_{i=1}^{m}\det(X_{(i, i+1)})\text{ (upto sign)} ,$  where
\begin{align}
X_{(i, i+1)}&=\left(
\begin{array}{ccccc}
x^{(1,1)}_{i}-x^{(1,1)}_{i+1}&x^{(1,2)}_{i}-x^{(1,2)}_{i+1}& \cdots &x^{(1,N)}_{i}-x^{(1,N)}_{i+1}\\
x^{(2,1)}_{i}-x^{(2,1)}_{i+1}&x^{(2,2)}_{i}-x^{(2,2)}_{i+1}&\cdots&x^{(2,N)}_{i}-x^{(2,N)}_{i+1}\\
\vdots&\vdots&\ddots&\vdots\\
x^{(N,1)}_{i}-x^{(N,1)}_{i+1}&x^{(N,2)}_{i}-x^{(N,2)}_{i+1}&\cdots&x^{(N,N)}_{i}-x^{(N,N)}_{i+1}	
\end{array}
\right).
\label{hookmat:1}
\end{align}
Moreover, we give  combinatorial explanation of these product formulas.
Throughout this paper we denote $i\times j$ zero matrix by
$O_{i, j}, R_i\leftrightarrow R_j (C_i\leftrightarrow C_j)$ denotes the row (column) interchange between $i^{\text{th}}$ and $j^{\text{th}}$ row (column) of a matrix, we denote the set $\{1, 2, \cdots, m\}$  by $[m].$ 

Now, let us briefly summarize the content. In Section $2,$ we consider block hook matrices containing same shape hooks and derive the product formulas for the determinants of these matrices. In Section $3,$ we focus on the block hook  matrices formed by hooks of two different shapes and we show that the determinant of these matrices admit nice product formulas. In Section $4,$ we deal with block hook matrices containing hooks of four different shapes and find similar product formulas. In Section $5,$ we give  combinatorial interpretations of factorization property of block hook determinants.      
\section{Block matrices containing hooks of same shape}
In this section, we evaluate the determinants of block hook matrices, which are block matrices, in which each block is a hook and shape of all hooks are same. Now, we introduce an $N\times N$ block matrix $A(N,m)$ in the following way;
\begin{align}\label{hookmat:3}
A(N,m)=(A_{ij}), \text{ where } 1\leq i, j\leq N \text{ and each block } A_{ij}=A_m(x^{(i,j)}_{1}, \cdots, x^{(i,j)}_{m}).
\end{align}
So $A(N,m)$ is a hook matrix of order $Nm.$ For example,
\begin{align}\label{A_{3,2}hookmat:5}
A(3,2)=\renewcommand\arraystretch{1.6}\mleft(
\begin{array}{cc|cc|cc}
x^{(1,1)}_{2} & x^{(1,1)}_{2} & x^{(1,2)}_{2}&x^{(1,2)}_{2}&x^{(1,3)}_{2}&x^{(1,3)}_{2} \\
x^{(1,1)}_{2} & x^{(1,1)}_{1} & x^{(1,2)}_{2}&x^{(1,2)}_{1}&x^{(1,3)}_{2}&x^{(1,3)}_{1} \\
\hline
x^{(2,1)}_{2} & x^{(2,1)}_{2} & x^{(2,2)}_{2}&x^{(2,2)}_{2}&x^{(2,3)}_{2}&x^{(2,3)}_{2} \\
x^{(2,1)}_{2} & x^{(2,1)}_{1} & x^{(2,2)}_{2}&x^{(2,2)}_{1}&x^{(2,3)}_{2}&x^{(2,3)}_{1} \\
\hline
x^{(3,1)}_{2} & x^{(3,1)}_{2} & x^{(3,2)}_{2}&x^{(3,2)}_{2}&x^{(3,3)}_{2}&x^{(3,3)}_{2} \\
x^{(3,1)}_{2} & x^{(3,1)}_{1} & x^{(3,2)}_{2}&x^{(3,2)}_{1}&x^{(3,3)}_{2}&x^{(3,3)}_{1} 
\end{array}
\mright)
\end{align}
is a block hook matrix, whose each block is a hook like \eqref{hookmat:2a}. 
\begin{theorem}\label{hookmatA}
Let $A(N,m)$ be a block hook matrix of order $Nm$ defined as \eqref{hookmat:3}. Then  \[\det(A(N,m))=\prod\limits_{i=1}^{m}\det(X_{(i, i+1)}),\] where $X_{(i, i+1)}$ is defined as \eqref{hookmat:1} and $x^{(i,j)}_{m+1}=0,$  for all $1\leq i, j\leq N.$
\end{theorem}
\begin{proof}
We prove the theorem by applying induction on $m.$ For $m=1,$ clearly the matrix \[A(N,1)=\renewcommand\arraystretch{1.6}\mleft(
\begin{array}{c|c|c|c}
x^{(1,1)}_{1} & x^{(1,2)}_{1} & \cdots&x^{(1,N)}_{1} \\
\hline
x^{(2,1)}_{1} & x^{(2,2)}_{1} & \cdots&x^{(2,N)}_{1} \\
\hline
\vdots&\vdots&\ddots&\vdots\\
\hline
x^{(N,1)}_{1} & x^{(N,2)}_{1} & \cdots&x^{(N,N)}_{1} 
\end{array}
\mright)=X_{(1,2)}, \text{ with } x^{(i,j)}_{2}=0, \text{ for all } 1\leq i, j\leq N .\]  So, the base case is true. Suppose the result is true for all such matrices $A(N,(m-1)).$ Now to complete the proof we want to go through the following row operations on the matrix $A(N,m);$ \[R_i'=R_{i}-R_{i-1}, \text{ for } i=m, 2m, \cdots, Nm.\]
Then the matrix, we obtain thereby is as follows:

All the rows of each block  $A_{ij}$ remain unchanged except the last row. The last row of each block $A_{ij}$ looks like  $0, 0, \cdots, (x^{(i,j)}_{1}-x^{(i,j)}_{2}).$  So for  $\ell=1, 2, \cdots, N,$ the $(\ell m)^\text{{th}}$ row of the matrix $A(N,m)$ is \[0, 0, \cdots, (x^{(\ell,1)}_{1}-x^{(\ell,1)}_{2}), 0, 0, \cdots,  (x^{(\ell,2)}_{1}-x^{(\ell,2)}_{2}), \cdots, 0, 0, \cdots, (x^{(\ell,N)}_{1}-x^{(\ell,N)}_{2}).\]
Now we apply the successive row interchange on the resulting matrix (after performing row operations on the matrix $A(N,m)$) to arrange the rows in the following order;
\[R_m, R_{2m}, \cdots, R_{Nm},  R_1,  \cdots, R_{m-1}, R_{m+1}, \cdots, R_{2m-1}, \cdots, R_{(N-1)m+1}, \cdots, R_{Nm-1}.\] Then by successive column interchange we arrange the columns in the following order;
\[C_m, C_{2m}, \cdots, C_{Nm},  C_{1},  \cdots, C_{m-1}, C_{m+1}, \cdots, C_{2m-1}, \cdots, C_{(N-1)m+1}, \cdots, C_{Nm-1}.\] And finally we obtain the matrix of the form 
\[
\renewcommand\arraystretch{1.3}
\mleft(
\begin{array}{c|c}
X_{(1, 2)} & O_{N, (Nm-N)} \\
\hline
\ast &  A(N,(m-1))
\end{array}
\mright),
\]
%\[\left(
%\begin{array}{ccc}
%X_{(1, 2)} & \vdots & O_{N, (Nm-N)} \\
%\cdots & \cdots & \cdots \\
% * & \vdots & A_{N, (m-1)} \\
%\end{array}
%\right),\] 
where $ A(N,(m-1))$ is block matrix and each block is  $A_{m-1}(x^{(i,j)}_{2}, \cdots, x^{(i,j)}_{m}).$ Therefore, $A(N,(m-1))$ is a block hook matrix of order $N(m-1).$ 
Now, without doubt we can write
\begin{align}\label{partition of matrix}
\det(A(N,m))=\renewcommand\arraystretch{1.3}
\mleft(
\begin{array}{c|c}
X_{(1, 2)} & O_{N, (Nm-N)} \\
\hline
\ast & A(N,(m-1))
\end{array}
\mright).
\end{align} 
Again by the Laplace expansion 
\begin{align}\label{laplace exp}
\det\renewcommand\arraystretch{1.3}
\mleft(
\begin{array}{c|c}
X_{(1, 2)} & O_{N, (Nm-N)} \\
\hline
\ast &  A(N,(m-1))
\end{array}
\mright)&=\det(X_{(1, 2)})\times \det(A(N,(m-1))).
\end{align}
Now we set $x^{(i,j)}_k=y^{(i,j)}_{k-1}, \text{ for } k=2, 3, \cdots, m$ and $1\leq i, j\leq N.$ Then $A_{m-1}(x^{(i,j)}_{2}, \cdots, x^{(i,j)}_{m})=A_{m-1}(y^{(i,j)}_{1}, \cdots, y^{(i,j)}_{m-1}).$ Again by inductive hypothesis we can write
\begin{equation}\label{replac Y} 
\det(A(N,(m-1)))=\prod\limits_{i=1}^{m-1}\det(Y_{(i, i+1)}),
\end{equation} 
where $Y_{(i, i+1)}$ is the matrix obtained by putting $x^{(r,s)}_{i+1}-x^{(r,s)}_{i+2}=y^{(r,s)}_{i}-y^{(r,s)}_{i+1}$ in $X_{(i+1,i+2)}$ for all $ 1\leq r, s\leq N,$ and $y^{(i,j)}_{m}=x^{(i,j)}_{m+1}=0.$ Clearly
\begin{align}\label{X=Y}
\prod\limits_{i=1}^{m-1}\det(Y_{(i, i+1)})=\prod\limits_{i=2}^{m}\det(X_{(i, i+1)}).
\end{align}
Therefore, by \eqref{partition of matrix}, \eqref{laplace exp}, \eqref{replac Y} and \eqref{X=Y} we get \[\det(A(N,m))=\prod\limits_{i=1}^{m}\det(X_{(i, i+1)}).\]
\end{proof}
Here we think about the determinantal formulas for the block hook matrix formed by the hook matrices like $B_{m}(x_1, \cdots, x_m).$ Now we define an $N\times N$ block matrix $B(N, m)$ in the following way;
\begin{align}\label{hookmat B}
B(N, m)=(B_{ij}), \text{ where } 1\leq i, j\leq N \text{ and } B_{ij}=B_m(x^{(i,j)}_{1}, \cdots, x^{(i,j)}_{m}).
\end{align}
So $B(N, m)$ is a block hook matrix of order $Nm.$ For example,
\begin{align}\label{B_{3,2}hookmat:5}
B(3, 2)=\renewcommand\arraystretch{1.6}\mleft(
\begin{array}{cc|cc|cc}
x^{(1,1)}_{1} & x^{(1,1)}_{2} & x^{(1,2)}_{1}&x^{(1,2)}_{2}&x^{(1,3)}_{1}&x^{(1,3)}_{2} \\
x^{(1,1)}_{2} &x^{(1,1)}_{2} & x^{(1,2)}_{2}&x^{(1,2)}_{2}&x^{(1,3)}_{2}&x^{(1,3)}_{2} \\
\hline
x^{(2,1)}_{1} & x^{(2,1)}_{2} & x^{(2,2)}_{1}&x^{(2,2)}_{2}&x^{(2,3)}_{1}&x^{(2,3)}_{2} \\
x^{(2,1)}_{2} & x^{(2,1)}_{2} & x^{(2,2)}_{2}&x^{(2,2)}_{2}&x^{(2,3)}_{2}&x^{(2,3)}_{2} \\
\hline
x^{(3,1)}_{1} & x^{(3,1)}_{2} & x^{(3,2)}_{1}&x^{(3,2)}_{2}&x^{(3,3)}_{1}&x^{(3,3)}_{2} \\
x^{(3,1)}_{2} & x^{(3,1)}_{2} & x^{(3,2)}_{2}&x^{(3,2)}_{2}&x^{(3,3)}_{2}&x^{(3,3)}_{2} 
\end{array}
\mright). 
\end{align}
\begin{theorem}\label{hookmatB}
Let $B(N, m)$ be a block hook matrix of order $Nm$ defined as \eqref{hookmat B}. Then \[\det(B(N, m))=\prod\limits_{i=1}^{m}\det(X_{(i,i+1)}),\] where $X_{(i,i+1)}$ is defined by \eqref{hookmat:1} and $x^{(i,j)}_{m+1}=0,$  for all $1\leq i, j \leq N.$	
\end{theorem}
\begin{proof}
We proof this theorem by performing some row and column operations on $A(N,m),$ so that the resulting matrix is $B(N,m).$ In fact, we perform the following row and column operations; 
\begin{equation}\label{important row operation1}
R_{km+i}\leftrightarrow R_{km+(m-i+1)}\text{ and } C_{km+i}\leftrightarrow C_{km+(m-i+1)},  \text{ where } k=0, 1, \cdots, (N-1) \text{ and }
\end{equation}
\begin{equation}\label{important row operation1 condition}
i=
\begin{cases}
1, 2, \cdots, \frac{m}{2}, & \text{if } m \text{ is even  } \\
1, 2, \cdots, \frac{m+1}{2}-1, & \text{ if } m \text{ is odd  }.\
\end{cases}
\end{equation}
Obviously, the resulting matrix is $B(N,m).$ Hence the result.
\end{proof}
Now we evaluate the determinant of the block hook matrix $C(N,m)$ formed by the hook matrices of the form $C_m(x_1, \cdots,  x_m).$ Let us define the matrix $C(N,m)$ as follows: 
\begin{align}\label{hook9}
C(N,m)=(C_{ij}), \text{ where }   C_{ij}=C_{m}(x^{(i,j)}_{1}, \cdots , x^{(i,j)}_{m})  \text{ and }  1\leq i, j\leq N.
\end{align}
For example,
\begin{align}\label{C_{3,2}hookmat:1}
C(3,2)=\renewcommand\arraystretch{1.6}\mleft(
\begin{array}{cc|cc|cc}
x^{(1,1)}_{2} & x^{(1,1)}_{1} & x^{(1,2)}_{2}&x^{(1,2)}_{1}&x^{(1,3)}_{2}&x^{(1,3)}_{1} \\
x^{(1,1)}_{2} & x^{(1,1)}_{2} & x^{(1,2)}_{2}&x^{(1,2)}_{2}&x^{(1,3)}_{2}&x^{(1,3)}_{2} \\
\hline
x^{(2,1)}_{2} & x^{(2,1)}_{1} & x^{(2,2)}_{2}&x^{(2,2)}_{1}&x^{(2,3)}_{2}&x^{(2,3)}_{1} \\
x^{(2,1)}_{2} & x^{(2,1)}_{2} & x^{(2,2)}_{2}&x^{(2,2)}_{2}&x^{(2,3)}_{2}&x^{(2,3)}_{2} \\
\hline
x^{(3,1)}_{2} & x^{(3,1)}_{1} & x^{(3,2)}_{2}&x^{(3,2)}_{1}&x^{(3,3)}_{2}&x^{(3,3)}_{1} \\
x^{(3,1)}_{2} & x^{(3,1)}_{2} & x^{(3,2)}_{2}&x^{(3,2)}_{2}&x^{(3,3)}_{2}&x^{(3,3)}_{2} 
\end{array}
\mright).
\end{align} 
\begin{theorem}\label{hooktypeC}
Let $C(N,m)$ be a block hook matrix of order $Nm$ defined as \eqref{hook9}. Then \[\det(C(N,m))=
\begin{cases}
(-1)^{\frac{Nm}{2}}\times\prod\limits_{i=1}^{m}\det(X_{(i, i+1)}), & \text{if } m \text{ is even  } \\
(-1)^{\frac{N(m-1)}{2}}\times\prod\limits_{i=1}^{m}\det(X_{(i, i+1)}), & \text{if } m \text{ is odd,  }
\end{cases}\] where $X_{(i, i+1)}$ is defined by \eqref{hookmat:1} and $x^{(i,j)}_{m+1}=0$  for all $1\leq i, j \leq N.$ 
\end{theorem}
\begin{proof}
If we apply the same row operations on the matrix $A(N,m)$ as in the proof of Theorem \ref{hookmatB}, we get the matrix $C(N,m).$ Hence the theorem.
\end{proof}
In this place we want to establish the determinantal formula for the block hook matrix $D(N,m),$ where $D(N,m)$ is an $Nm$ ordered matrix defined as follows: 
\begin{align}\label{hook matrix D}
\text{ Now, } D(N,m)=(D_{ij}), \text{ where } D_{ij}=D_{m}(x^{(i,j)}_{1},  \cdots, x^{(i,j)}_{m}) \text{ and }  1\leq i, j\leq N.
\end{align}
For example, 
\begin{align}\label{D_{3,2}hookmat:1}
D(3,2)=\renewcommand\arraystretch{1.6}\mleft(
\begin{array}{cc|cc|cc}
x^{(1,1)}_{2} & x^{(1,1)}_{2} & x^{(1,2)}_{2}&x^{(1,2)}_{2}&x^{(1,3)}_{2}&x^{(1,3)}_{2} \\
x^{(1,1)}_{1} & x^{(1,1)}_{2} & x^{(1,2)}_{1}&x^{(1,2)}_{2}&x^{(1,3)}_{1}&x^{(1,3)}_{2} \\
\hline
x^{(2,1)}_{2} & x^{(2,1)}_{2} & x^{(2,2)}_{2}&x^{(2,2)}_{2}&x^{(2,3)}_{2}&x^{(2,3)}_{2} \\
x^{(2,1)}_{1} & x^{(2,1)}_{2} & x^{(2,2)}_{1}&x^{(2,2)}_{2}&x^{(2,3)}_{1}&x^{(2,3)}_{2} \\
\hline
x^{(3,1)}_{2} & x^{(3,1)}_{2} & x^{(3,2)}_{2}&x^{(3,2)}_{2}&x^{(3,3)}_{2}&x^{(3,3)}_{2} \\
x^{(3,1)}_{1} & x^{(3,1)}_{2} & x^{(3,2)}_{1}&x^{(3,2)}_{2}&x^{(3,3)}_{1}&x^{(3,3)}_{2} 
\end{array}
\mright).
\end{align}
\begin{theorem}\label{hooktypeD}
Let $D(N,m)$ be a block hook  matrix of order $Nm$ defined as \eqref{hook matrix D}. Then \[\det(D(N,m))=\begin{cases}
(-1)^{\frac{Nm}{2}}\times\prod\limits_{i=1}^{m}\det(X_{(i, i+1)}), & \text{if } m \text{ is even  } \\
(-1)^{\frac{N(m-1)}{2}}\times\prod\limits_{i=1}^{m}\det(X_{(i, i+1)}), & \text{if } m \text{ is odd,  }
\end{cases}\] where $X_{(i, i+1)}$ is defined by \eqref{hookmat:1} and $x^{(i,j)}_{m+1}=0$  for all $1\leq i, j \leq N.$ 
\end{theorem}
\begin{proof}
Applying the same column operations as in the proof of Theorem \ref{hookmatB} on the matrix $A(N,m),$ we get the matrix $D(N,m).$ This proves the theorem.
\end{proof}
\section{block matrices containing hooks of two different shapes}
In this section, we present determinantal formulas for the block hook matrices containing hooks of two different shapes. In fact, we construct four different such block  matrices and evaluate their determinants. Let $E(N,m)$ be a block hook matrix defined in the following way;
\begin{equation} \label{hook  matrix E}
E(N,m)=(E_{ij})_{N\times N}, \text{ where }E_{ij}=
\begin{cases}
C_{m}(x^{(i,j)}_{1}, \cdots, x^{(i,j)}_{m}), & \text{if } i \text{ is even  } \\
A_m(x^{(i,j)}_{1}, \cdots, x^{(i,j)}_{m}), & \text{ if } i \text{ is odd  }.
\end{cases}
\end{equation}
For example, 
 \begin{align}\label{E_{3,2}hookmat:1}
 E(3,2)=\renewcommand\arraystretch{1.6}\mleft(
 \begin{array}{cc|cc|cc}
 x^{(1,1)}_{2} &x^{(1,1)}_{2} & x^{(1,2)}_{2}&x^{(1,2)}_{2}&x^{(1,3)}_{2}&x^{(1,3)}_{2} \\
 x^{(1,1)}_{2} & x^{(1,1)}_{1} & x^{(1,2)}_{2}&x^{(1,2)}_{1}&x^{(1,3)}_{2}&x^{(1,3)}_{1} \\
 \hline
x^{(2,1)}_{2} & x^{(2,1)}_{1} & x^{(2,2)}_{2}&x^{(2,2)}_{1}&x^{(2,3)}_{2}&x^{(2,3)}_{1} \\
 x^{(2,1)}_{2} &x^{(2,1)}_{2} & x^{(2,2)}_{2}&x^{(2,2)}_{2}&x^{(2,3)}_{2}&x^{(2,3)}_{2} \\
 \hline
 x^{(3,1)}_{2} & x^{(3,1)}_{2} & x^{(3,2)}_{2}&x^{(3,2)}_{2}&x^{(3,3)}_{2}&x^{(3,3)}_{2} \\
 x^{(3,1)}_{2} & x^{(3,1)}_{1} & x^{(3,2)}_{2}&x^{(3,2)}_{1}&x^{(3,3)}_{2}&x^{(3,3)}_{1} 
 \end{array}
 \mright).
 \end{align}
 \begin{theorem}\label{hookdet E}
 Let $E(N,m)$ be a block hook matrix of order $Nm$ defined as \eqref{hook matrix E}. Then  
 \[\det(E(N,m))=\begin{cases}
 (-1)^{\frac{Nm}{4}}\prod\limits_{i=1}^{m}\det(X_{(i, i+1)}), & \text{ if }  N \text{  is even,  } m \text{ is even }\\
 (-1)^{\frac{N(m-1)}{4}}\prod\limits_{i=1}^{m}\det(X_{(i, i+1)}), & \text{ if }  N \text{  is even,  } m \text{ is odd } \\
 (-1)^{\frac{(N-1)m}{4}}\prod\limits_{i=1}^{m}\det(X_{(i, i+1)}), & \text{ if }  N \text{  is odd,  } m \text{ is even }\\
 (-1)^{\frac{(N-1)(m-1)}{4}}\prod\limits_{i=1}^{m}\det(X_{(i, i+1)}), & \text{ if }  N \text{  is odd,  } m \text{ is odd },
 \end{cases}
 \] where $X_{(i, i+1)}$ is defined as \eqref{hookmat:1} and $x^{(i,j)}_{m+1}=0$  for all $1 \leq i, j \leq N.$	
 \end{theorem}	
\begin{proof}
Here we apply the following elementary row operations  on the matrix $A(N,m);$ 
\begin{equation}\label{important row operation2}
R_{km+i}\leftrightarrow R_{km+(m-i+1)},  \text{ where }
\end{equation}
\begin{equation}\label{important row operation2 condition}
 i=
\begin{cases}
1, 2, \cdots, \frac{m}{2}, & \text{if } m \text{ is even  } \\
1, 2, \cdots, \frac{m+1}{2}-1, & \text{ if } m \text{ is odd  }\
\end{cases}
\text{ and } k=
\begin{cases} 
1, 3, \cdots, N-1, & \text{if } N \text{ is even  } \\
1, 3, \cdots, N-2, & \text{ if } N \text{ is odd  }.\
\end{cases}
\end{equation}
Evidently the resulting matrix is $E(N,m).$ Hence the theorem.  
\end{proof}
Let us introduce a block hook matrix $E'(N,m)$ in the following way;
\begin{equation}\label{hook acute{E_{ij}}1}
E'(N,m)=(E'_{ij})_{N\times N}, \text{ where }
E'_{ij}=
\begin{cases}
C_{m}(x^{(i,j)}_{1},  \cdots, x^{(i,j)}_{m}), & \text{if } i \text{ is odd  } \\
A_m(x^{(i,j)}_{1},  \cdots, x^{(i,j)}_{m}), & \text{ if } i \text{ is even }.
\end{cases}
\end{equation} 
  For example,
   \begin{align}\label{acuteE_{3,2}hookmat:1}
   E'(3,2)=\renewcommand\arraystretch{1.6}\mleft(
   \begin{array}{cc|cc|cc}
   x^{(1,1)}_{2} & x^{(1,1)}_{1} & x^{(1,2)}_{2}&x^{(1,2)}_{1}&x^{(1,3)}_{2}&x^{(1,3)}_{1} \\
   x^{(1,1)}_{2} & x^{(1,1)}_{2} & x^{(1,2)}_{2}&x^{(1,2)}_{2}&x^{(1,3)}_{2}&x^{(1,3)}_{2} \\
   \hline
  x^{(2,1)}_{2} & x^{(2,1)}_{2} & x^{(2,2)}_{2}&x^{(2,2)}_{2}&x^{(2,3)}_{2}&x^{(2,3)}_{2} \\
   x^{(2,1)}_{2} & x^{(2,1)}_{1}& x^{(2,2)}_{2}&x^{(2,2)}_{1}&x^{(2,3)}_{2}&x^{(2,3)}_{1} \\
   \hline
   x^{(3,1)}_{2} & x^{(3,1)}_{1} & x^{(3,2)}_{2}&x^{(3,2)}_{1}&x^{(3,3)}_{2}&x^{(3,3)}_{1} \\
   x^{(3,1)}_{2} & x^{(3,1)}_{2} & x^{(3,2)}_{2}&x^{(3,2)}_{2}&x^{(3,3)}_{2}&x^{(3,3)}_{2} 
   \end{array}
   \mright).
   \end{align}
 \begin{theorem}\label{th acute E}
 Let $E'(N,m)$ be a block hook matrix of order $Nm$ defined as \eqref{hook acute{E_{ij}}1}. Then \[\det(E'(N,m))=\begin{cases}
 (-1)^{\frac{Nm}{4}}\prod\limits_{i=1}^{m}\det(X_{(i, i+1)}), & \text{ if }  N \text{  is even,  } m \text{ is even }\\
 (-1)^{\frac{N(m-1)}{4}}\prod\limits_{i=1}^{m}\det(X_{(i, i+1)}), & \text{ if }  N \text{  is even,  } m \text{ is odd } \\
 (-1)^{\frac{(N-1)m}{4}}\prod\limits_{i=1}^{m}\det(X_{(i, i+1)}), & \text{ if }  N \text{  is odd,  } m \text{ is even }\\
 (-1)^{\frac{(N-1)(m-1)}{4}}\prod\limits_{i=1}^{m}\det(X_{(i, i+1)}), & \text{ if }  N \text{  is odd,  } m \text{ is odd },
 \end{cases}
 \] where $X_{(i, i+1)}$ is defined as \eqref{hookmat:1} and $x^{(i,j)}_{m+1}=0$  for all $1 \leq i, j \leq N.$
  \end{theorem}
\begin{proof}
Applying the same row operations described by \eqref{important row operation2} and \eqref{important row operation2 condition} on the matrix $C(N,m),$ we get the matrix $E'(N,m).$ This completes the proof. 
\end{proof}
Here, we define another block hook matrix $F(N,m)$ as follows:
\begin{equation}\label{hook F 1}
F(N,m)=(F_{ij})_{N\times N}, \text{ where }F_{ij}=
\begin{cases}
B_{m}(x^{(i,j)}_{1}, \cdots,  x^{(i,j)}_{m}), & \text{if } i \text{ is odd } \\
D_m(x^{(i,j)}_{1},  \cdots, x^{(i,j)}_{m}), & \text{ if } i \text{ is even }.
\end{cases}
\end{equation}
For example:
\begin{align}\label{F_{3,2}hookmat:1}
F(3,2)=\renewcommand\arraystretch{1.6}\mleft(
\begin{array}{cc|cc|cc}
x^{(1,1)}_{1} & x^{(1,1)}_{2} & x^{(1,2)}_{1}&x^{(1,2)}_{2}&x^{(1,3)}_{1}&x^{(1,3)}_{2} \\
x^{(1,1)}_{2} & x^{(1,1)}_{2} & x^{(1,2)}_{2}&x^{(1,2)}_{2}&x^{(1,3)}_{2}&x^{(1,3)}_{2} \\
\hline
x^{(2,1)}_{2} & x^{(2,1)}_{2} & x^{(2,2)}_{2}&x^{(2,2)}_{2}&x^{(2,3)}_{2}&x^{(2,3)}_{2} \\
x^{(2,1)}_{1} & x^{(2,1)}_{2} & x^{(2,2)}_{1}&x^{(2,2)}_{2}&x^{(2,3)}_{1}&x^{(2,3)}_{2} \\
\hline
x^{(3,1)}_{1} & x^{(3,1)}_{2} & x^{(3,2)}_{1}&x^{(3,2)}_{2}&x^{(3,3)}_{1}&x^{(3,3)}_{2} \\
x^{(3,1)}_{2} & x^{(3,1)}_{2} & x^{(3,2)}_{2}&x^{(3,2)}_{2}&x^{(3,3)}_{2}&x^{(3,3)}_{2} 
\end{array}
\mright).
\end{align}
\begin{theorem}\label{hook mat F thm}
Let $F(N,m)$ be a block hook matrix of order $Nm$ defined as \eqref{hook F 1}. Then \[\det(F(N,m))=\begin{cases}
(-1)^{\frac{Nm}{4}}\prod\limits_{i=1}^{m}\det(X_{(i, i+1)}), & \text{ if }  N \text{  is even,  } m \text{ is even }\\
(-1)^{\frac{N(m-1)}{4}}\prod\limits_{i=1}^{m}\det(X_{(i, i+1)}), & \text{ if }  N \text{  is even,  } m \text{ is odd } \\
(-1)^{\frac{(N-1)m}{4}}\prod\limits_{i=1}^{m}\det(X_{(i, i+1)}), & \text{ if }  N \text{  is odd,  } m \text{ is even }\\
(-1)^{\frac{(N-1)(m-1)}{4}}\prod\limits_{i=1}^{m}\det(X_{(i, i+1)}), & \text{ if }  N \text{  is odd,  } m \text{ is odd },
\end{cases}
\] where $X_{(i, i+1)}$ is defined as \eqref{hookmat:1} and $x^{(i,j)}_{m+1}=0$  for all $1 \leq i, j \leq N.$	
\end{theorem}
\begin{proof}
Here we apply the same row operations described by \eqref{important row operation2} and \eqref{important row operation2 condition} on the matrix $B(N,m).$ Evidently the resulting matrix is $F(N,m).$ Hence the theorem.
\end{proof}
Let us define another block hook matrix $F'(N,m)$ as follows:  
\begin{equation}\label{acute F hook 1}
F'(N,m)=(F'_{ij})_{N\times N}, \text{ where } F'_{ij}=
\begin{cases}
D_{m}(x^{(i,j)}_{1},  \cdots, x^{(i,j)}_{m}), & \text{if } i \text{ is odd  } \\
B_m(x^{(i,j)}_{1},  \cdots,  x^{(i,j)}_{m}), & \text{ if } i \text{ is even }.
\end{cases}
\end{equation}
For example,
\begin{align}\label{acute F_{3,2}hookmat:1}
F'(3,2)=\renewcommand\arraystretch{1.6}\mleft(
\begin{array}{cc|cc|cc}
x^{(1,1)}_{2} & x^{(1,1)}_{2} & x^{(1,2)}_{2}&x^{(1,2)}_{2}&x^{(1,3)}_{2}&x^{(1,3)}_{2} \\
x^{(1,1)}_{1} & x^{(1,1)}_{2} & x^{(1,2)}_{1}&x^{(1,2)}_{2}&x^{(1,3)}_{1}&x^{(1,3)}_{2} \\
\hline
x^{(2,1)}_{1} & x^{(2,1)}_{2} & x^{(2,2)}_{1}&x^{(2,2)}_{2}&x^{(2,3)}_{1}&x^{(2,3)}_{2} \\
x^{(2,1)}_{2} & x^{(2,1)}_{2} & x^{(2,2)}_{2}&x^{(2,2)}_{2}&x^{(2,3)}_{2}&x^{(2,3)}_{2} \\
\hline
x^{(3,1)}_{2} & x^{(3,1)}_{2} & x^{(3,2)}_{2}&x^{(3,2)}_{2}&x^{(3,3)}_{2}&x^{(3,3)}_{2} \\
x^{(3,1)}_{1} & x^{(3,1)}_{2} & x^{(3,2)}_{1}&x^{(3,2)}_{2}&x^{(3,3)}_{1}&x^{(3,3)}_{2}
\end{array}
\mright).
\end{align}
\begin{theorem}\label{acute F hook thm}
Let $F'(N,m)$ be a block hook matrix of order $Nm$ defined as \eqref{acute F hook 1}. Then \[\det(F'(N,m))=\begin{cases}
(-1)^{\frac{Nm}{4}}\prod\limits_{i=1}^{m}\det(X_{(i, i+1)}), & \text{ if }  N \text{  is even,  } m \text{ is even }\\
(-1)^{\frac{N(m-1)}{4}}\prod\limits_{i=1}^{m}\det(X_{(i, i+1)}), & \text{ if }  N \text{  is even,  } m \text{ is odd } \\
(-1)^{\frac{(N-1)m}{4}}\prod\limits_{i=1}^{m}\det(X_{(i, i+1)}), & \text{ if }  N \text{  is odd,  } m \text{ is even }\\
(-1)^{\frac{(N-1)(m-1)}{4}}\prod\limits_{i=1}^{m}\det(X_{(i, i+1)}), & \text{ if }  N \text{  is odd,  } m \text{ is odd },
\end{cases}
\] where $X_{(i, i+1)}$ is defined as \eqref{hookmat:1} and $x^{(i,j)}_{m+1}=0$  for all $1 \leq i, j \leq N.$
\end{theorem}
\begin{proof}
In this case also we apply the same row operations described by \eqref{important row operation2} and \eqref{important row operation2 condition} on the matrix $D(N,m),$ and we get the matrix $F'(N,m).$ This completes the proof.  
\end{proof}
\section{block matrices containing hooks of four different shapes}
In this section, we are going to derive the determinantal formulas of block hook matrices, in which the blocks are suitable combination of four different shapes of hooks. In particular, here we construct two such block matrices and show that the determinants of these matrices are also product of determinants of matrices $X_{(i, i+1)}, (i=1, 2, \cdots, m)$ (defined as \eqref{hookmat:1}). So    		
let us define a block hook matrix $G(N,m)$ as follows:
\begin{equation}\label{hok 4type 1}
G(N,m)= (G_{ij})_{N\times N}, \text{ where } G_{ij}= 
 \begin{cases}
 B_m(x^{(i,j)}_{1}, \cdots, x^{(i,j)}_{m}), & \text{if } i,  j \text{ are odd }\\
 C_m(x^{(i,j)}_{1}, \cdots, x^{(i,j)}_{m}),    & \text{if }  i \text{ is odd, } j \text{ is  even }\\
 D_m(x^{(i,j)}_{1}, \cdots, x^{(i,j)}_{m}), &\text{if } i \text{ is even, } j \text{ is odd }\\
 A_m(x^{(i,j)}_{1}, \cdots, x^{(i,j)}_{m}), &\text{if } i,  j \text{ are even }.
 \end{cases}
\end{equation}
For example,
\begin{align}\label{ G_{3,2}hookmat:1}
G(3,2)=\renewcommand\arraystretch{1.6}\mleft(
\begin{array}{cc|cc|cc}
x^{(1,1)}_{1} & x^{(1,1)}_{2} & x^{(1,2)}_{2}&x^{(1,2)}_{1}&x^{(1,3)}_{1}&x^{(1,3)}_{2} \\
x^{(1,1)}_{2} & x^{(1,1)}_{2} & x^{(1,2)}_{2}&x^{(1,2)}_{2}&x^{(1,3)}_{2}&x^{(1,3)}_{2} \\
\hline
x^{(2,1)}_{2} & x^{(2,1)}_{2} & x^{(2,2)}_{2}&x^{(2,2)}_{2}&x^{(2,3)}_{2}&x^{(2,3)}_{2} \\
x^{(2,1)}_{1} & x^{(2,1)}_{2} & x^{(2,2)}_{2}&x^{(2,2)}_{1}&x^{(2,3)}_{1}&x^{(2,3)}_{2} \\
\hline
x^{(3,1)}_{1} & x^{(3,1)}_{2}& x^{(3,2)}_{2}&x^{(3,2)}_{1}&x^{(3,3)}_{1}&x^{(3,3)}_{2} \\
x^{(3,1)}_{2} & x^{(3,1)}_{2} & x^{(3,2)}_{2}&x^{(3,2)}_{2}&x^{(3,3)}_{2}&x^{(3,3)}_{2} 
\end{array}
\mright).
\end{align}
\begin{theorem}\label{4 type hook them 1}
Let $G(N,m)$ be a block hook matrix of order $Nm$ defined as \eqref{hok 4type 1}. Then\[\det(G(N,m))=\prod_{i=1}^{m}\det(X_{(i, i+1)}),\] where $X_{(i, i+1)}$ is defined by \eqref{hookmat:1} and $x^{(i,j)}_{m+1}=0$  for all $ 1\leq i, j \leq N.$	
\end{theorem}
\begin{proof}
Here we apply the same row and column operations described by \eqref{important row operation1} and \eqref{important row operation1 condition} on $B(N,m).$ Clearly after the effect of these row and column operations we get the matrix $G(N,m).$ Consequently we get the result.    
\end{proof}
Let us introduce one more block hook matrix  $G'(N,m)$ containing four different shape of hooks as follows:
\begin{equation}\label{hook 4type acute}
G'(N,m)=(G'_{ij})_{N\times N}, \text{ where } G'_{ij}=
\begin{cases}
A_m(x^{(i,j)}_{1}, \cdots, x^{(i,j)}_{m}), & \text{if } i, j \text{ are odd } \\
D_m(x^{(i,j)}_{1}, \cdots, x^{(i,j)}_{m}), & \text{if }  i \text{ is odd,  } j \text{ is even } \\
C_m(x^{(i,j)}_{1}, \cdots, x^{(i,j)}_{m}), &\text{if } i \text{ is even }, j \text{ is odd }  \\
B_m(x^{(i,j)}_{1}, \cdots, x^{(i,j)}_{m}), &\text{if } i, j \text{ are even }.
\end{cases}
\end{equation}
For example,
\begin{align}\label{ acuteG_{3,2}hookmat:1}
G'(3,2)=\renewcommand\arraystretch{1.6}\mleft(
\begin{array}{cc|cc|cc}
x^{(1,1)}_{2} & x^{(1,1)}_{2} & x^{(1,2)}_{2}&x^{(1,2)}_{2}&x^{(1,3)}_{2}&x^{(1,3)}_{2} \\
x^{(1,1)}_{2} & x^{(1,1)}_{1} & x^{(1,2)}_{1}&x^{(1,2)}_{2}&x^{(1,3)}_{2}&x^{(1,3)}_{1} \\
\hline
x^{(2,1)}_{2} & x^{(2,1)}_{1} & x^{(2,2)}_{1}&x^{(2,2)}_{2}&x^{(2,3)}_{2}&x^{(2,3)}_{1} \\
x^{(2,1)}_{2} & x^{(2,1)}_{2} & x^{(2,2)}_{2}&x^{(2,2)}_{2}&x^{(2,3)}_{2}&x^{(2,3)}_{2} \\
\hline
x^{(3,1)}_{2} & x^{(3,1)}_{2} & x^{(3,2)}_{2}&x^{(3,2)}_{2}&x^{(3,3)}_{2}&x^{(3,3)}_{2} \\
x^{(3,1)}_{2} & x^{(3,1)}_{1} & x^{(3,2)}_{1}&x^{(3,2)}_{2}&x^{(3,3)}_{2}&x^{(3,3)}_{1} 
\end{array}
\mright).
\end{align}
\begin{theorem}\label{acute hook 4 type thm1}
Let $G'(N,m)$ be a block hook matrix of order $Nm$ defined as \eqref{hook 4type acute}. Then \[\det(G'(N,m))=\prod_{i=1}^{m}\det(X_{(i, i+1)}), \]	where $X_{(i, i+1)}$ is defined by \eqref{hookmat:1} and $x^{(i,j)}_{m+1}=0$  for all $1\leq i, j \leq N.$	
\end{theorem}
\begin{proof}
 Applying the same row and column operations depicted by  \eqref{important row operation1} and \eqref{important row operation1 condition} on the matrix $A(N,m),$ we get the matrix $G'(N,m).$ Hence the theorem. 
\end{proof}
\section{combinatorial interpretations of hook determinants}
Combinatorial interpretations of determinants can bring deeper understanding to their evaluations; this is especially true when the entries of a matrix have natural
graph theoretic descriptions \cite{18, 14, 41, 40, 11}. In this section, we give combinatorial interpretations of our main results (that are stated in previous sections) regarding hook determinants.  Before plunging into the proof, let us recall the celebrated ``Gessel-Lindstr\"om-Viennot'' lemma. See \cite{15, 41}, for details. For the sake of completeness, let us reproduce the lemma from \cite{15}. Let $\Gamma$ be a weighted, acyclic digraph. The vertex set and the edge set of the graph $\Gamma,$ denoted by $V(\Gamma)$ and $E(\Gamma)$ respectively. A path in  $\Gamma$ is a sequence of distinct vertices $V_1, V_2, \cdots, V_r$ such that $V_i, V_{i+1} (i=1, \cdots, r-1)$ is an edge directed from $V_i$ to $V_{i+1}.$  For simplicity we denote a path $V_1, V_2, \cdots, V_r$  by $V_1 V_2 \cdots V_r$ and hence an edge $V_i, V_{i+1}$ by $V_iV_{i+1}.$ The \emph{weight} of a path $P,$ denoted by $w(P)$ \text { is the product of weights of all edges involved in the path  } and the \emph{length} of a path $P,$ denoted by $\ell(P),$ is the number of edges involved in the path $P.$ Suppose that $U=\{ U_1, U_2, \cdots, U_n\}$ and $V=\{V_1, V_2, \cdots, V_n\}$ are two $n$-sets of vertices of $\Gamma$ (not necessarily disjoint). To $U$ and $V$, associate the \emph{path matrix} $M=(m_{ij})_{n\times n},$ where $m_{ij}=\sum\limits_{P:U_i\rightarrow V_j}w(P),$  $P:U_i\rightarrow V_j$ \text{ denotes a path from } $U_i \text{ to } V_j.$ A \emph{path system} from $U$ to $V$ is an ordered pair $(\mathcal{P}, \sigma),$ where $\sigma$ is a permutation of $n$ element set and $\mathcal{P}$ is a set of $n$ paths $P_i: U_i\rightarrow V_{\sigma(i)}.$ The sign of a path system $(\mathcal{P}, \sigma)$ is $\text{sgn} (\sigma).$ The \emph{weight} of $(\mathcal{P}, \sigma),$ denoted by $w(\mathcal{P}, \sigma)$ is $\prod\limits_{i=1}^nw(P_i)$. We call the path system \emph{vertex-disjoint} if no two paths have a common vertex. Let $VD_{\Gamma}$ be the family of vertex-disjoint path systems in the graph $\Gamma.$ Then the Gessel-Lindstr\"om-Viennot lemma is the following;
\begin{lemma}[Gessel-Lindstr\"om-Viennot lemma]\label{lgv lemma}
Let $\Gamma=(V(\Gamma), E(\Gamma))$ be a weighted acyclic digraph. Suppose $U=\{ U_1, U_2, \cdots, U_n\}, V=\{V_1, V_2, \cdots, V_n\}$ are two $n$-sets of vertices of $\Gamma$ (not necessarily disjoint) and $M$ be a path matrix from $U$ to $V.$ Then	
 \[\det(M)=\sum\limits_{(\mathcal{P}, \sigma)\in VD_{\Gamma}}\text{ sgn}(\mathcal{P}, \sigma)w(\mathcal{P}, \sigma).\]
\end{lemma}
Now we give combinatorial proofs of theorems stated in previous sections using Lemma \ref{lgv lemma}.
\begin{proof}[Combinatorial proof of Theorem \ref{elem hook them Am}]
Consider the acyclic weighted digraph in Figure \ref{fig:combinatorial proof of A_{m}}. We choose the sets $\{U_1, U_2, \cdots, U_m\}$ and $\{V_1, V_2, \cdots, V_m\}$ as the initial and terminal sets of vertices respectively. Clearly the associate path matrix is the matrix $A_m(x_1, \cdots, x_m)$ defined as \eqref{hookmat:2a}. Evidently, there is only one vertex disjoint path system $(\mathcal{P}, \text {Id}),$ where \text{Id} is the identity permutation of $m$ element set and $\mathcal{P}$ is a set of $m$ paths $P_i:U_i\rightarrow V_i (i=1, \cdots, m).$  Now applying Lemma \ref{lgv lemma} we get the theorem.
\end{proof}	
 \begin{figure}[H]
 	\tiny
 	\tikzstyle{ver}=[]
 	\tikzstyle{vert}=[circle, draw, fill=black!100, inner sep=0pt, minimum width=4pt]
 	\tikzstyle{vertex}=[circle, draw, fill=black!00, inner sep=0pt, minimum width=4pt]
 	\tikzstyle{edge} = [draw,thick,-]
 	%\tikzstyle{edge_style} = [draw=black, line width=2, ultra thick]
 	\tikzstyle{node_style} = [circle,draw=blue,fill=blue!20!,font=\sffamily\Large\bfseries]
 	\centering
 	\tikzset{->,>=stealth', auto,node distance=1cm,
 		thick,main node/.style={circle,draw,font=\sffamily\Large\bfseries}}
 	\tikzset{->-/.style={decoration={
 				markings,
 				mark=at position #1 with {\arrow{>}}},postaction={decorate}}}
 \begin{tikzpicture}[scale=1]
 	\tikzstyle{edge_style} = [draw=black, line width=2mm, ]
 	\tikzstyle{node_style} = [draw=blue,fill=blue!00!,font=\sffamily\Large\bfseries]
 	%\node (B1) at (0,-4)   {$\bf{B_1}$};
 	%\node (B2) at (2,-4)  {$\bf{B_2}$};
 	%	\node (B3) at (4,-4)   {$\bf{B_3}$};
 	%\node (B4) at (6,-4)   {$\bf{B_4}$};
 	\node (A1) at (0,0)   {$\bf{U_1}$};
 	\node (A2) at (3,0)   {$\bf{U_2}$};
 	\node (A3) at (6,0)  {$\bf{U_3}$};
 %	\node (A4) at (6,0)  {$\bf{U_4}$};
 	\node (A4) at (8.3,.1)  {$\bf{U_{m-1}}$};
 	\node (A4) at (11,.1)  {$\bf{U_m}$};
 	\node (A4) at (6.9,-.25)  {$\bf{\cdots}$};
 	\node (A4) at (7.3,-.25)  {$\bf{\cdots}$};
 	\node (B1) at (-.2,-2.4)   {$\bf{V_1}$};
 	\node (B1) at (3,-2.4)   {$\bf{V_2}$};
 	\node (B1) at (6,-2.4)   {$\bf{V_3}$};
 	%\node (B1) at (5.9,-2.4)   {$\bf{V_4}$};
 	\node (A4) at (8.3,-2.4)  {$\bf{V_{m-1}}$};
 	\node (A4) at (11,-2.4)  {$\bf{V_m}$};
 	\node (A4) at (6.9,-2.1)  {$\bf{\cdots}$};
 	\node (A4) at (7.3,-2.1)  {$\bf{\cdots}$};
 	%\node (A4) at (7.1,-1.2)  {$\bf{\vdots}$};
 	%\node (A4) at (7.1,-1.2)  {$\bf{\vdots}$};
 	%
 	\fill[black!100!](0,-.2) circle (.05);
 	\fill[black!100!](3.1,-.2) circle (.05);
 	\fill[black!100!](5.9,-.2) circle (.05);
 %	\fill[black!100!](6.1,-.2) circle (.05);
 	\fill[black!100!](8.1,-.2) circle (.05);
 	\fill[black!100!](11,-.2) circle (.05);
 	    \fill[black!100!](-.1,-2) circle (.05);
 		\fill[black!100!](3.1,-2) circle (.05);
 		\fill[black!100!](5.9,-2) circle (.05);
 		%\fill[black!100!](6.1,-2) circle (.05);
 		\fill[black!100!](8.1,-2) circle (.05);
 		\fill[black!100!](11,-2) circle (.05); 
 	\draw[->, line width=.2 mm] (10.8,-.2) -- (8.2,-.2);
 	%\draw[->, line width=.2 mm] (6,-.2) -- (4.1,-.2);
 	\draw[->, line width=.2 mm] (5.8,-.2) -- (3.2,-.2);
 	\draw[->, line width=.2 mm] (3,-.2) -- (.1,-.2);
 	\draw[->, line width=.2 mm] (-.1,-.3) -- (-.1,-1.9);
 	\draw[->, line width=.2 mm] (3.1,-.3) -- (3.1,-1.9);
 	\draw[->, line width=.2 mm] (5.9,-.3) -- (5.9,-1.9);
 	%\draw[->, line width=.2 mm] (6.2,-.2) -- (6.2,-1.9);
 	\draw[->, line width=.2 mm] (8,-.3) -- (8,-1.9);
 	\draw[->, line width=.2 mm] (11,-.3) -- (11,-1.9);
 	\draw[->, line width=.2 mm] (0,-2) -- (3,-2);
 	\draw[->, line width=.2 mm] (3.2,-2) -- (5.8,-2);
 	%\draw[->, line width=.2 mm] (4.2,-2) -- (6,-2);
 	\draw[->, line width=.2 mm] (8.3,-2) -- (10.8,-2);
 	%\draw[->, line width=.2 mm] (0,-2) -- (0,-3.8);
 	%\draw[->, line width=.2 mm] (2,-2) -- (2,-3.8);
 	%\draw[->, line width=.2 mm] (4,-2) -- (4,-3.8);
 	%\draw[->, line width=.2 mm] (6,-2) -- (6,-3.8);
 	%\draw (-1,-1) -- (2.2,-1) -- (2.2,2) -- (-1,2) -- (-1,-1);
 	\node (B1) at (-.4,-1)   {$\bf{x_m}$};
 \node (B1) at (2,-1)   {$\bf{(x_{m-1}-x_m)}$};
 \node (B1) at (4.7,-1)   {$\bf{(x_{m-2}-x_{m-1})}$};
 	\node (B1) at (7.2,-1)   {$\bf{(x_{2}-x_{3})}$};
 		\node (B1) at (10.2,-1)   {$\bf{(x_{1}-x_{2})}$};
 	\node (B1) at (1,.1)   {$\bf{1}$};
 	\node (B1) at (4.5,.1)   {$\bf{1}$};
 	%\node (B1) at (5,.1)   {$\bf{1}$};
 	\node (B1) at (1,.1)   {$\bf{1}$};
 	\node (B1) at (1,-2.3)   {$\bf{1}$};
 	\node (B1) at (4.5,-2.3)   {$\bf{1}$};
 	%\node (B1) at (5,-2.3)   {$\bf{1}$};
 	\node (B1) at (9.5,.15)   {$\bf{1}$};
 	\node (B1) at (9.5,-2.3)   {$\bf{1}$};
 	\end{tikzpicture}
 \caption{A weighted acyclic digraph $\Gamma_m,$ and the weight of each edge is described on figure.}
 	\label{fig:combinatorial proof of A_{m}}	
 \end{figure}
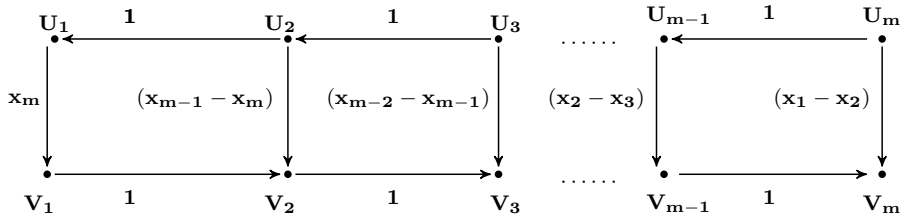
 \begin{rem}
In exactly the same way we can give combinatorial explanation of theorems \ref{elem hook thm B_m}, \ref{elem hook thm C_m} and \ref{elem hook thm D_m}. In those cases also we use $\Gamma_{m}$ as combinatorial object possibly permuting some of its vertices.
 \end{rem}
\begin{proof}[Combinatorial proof of Theorem \ref{hookmatA}]
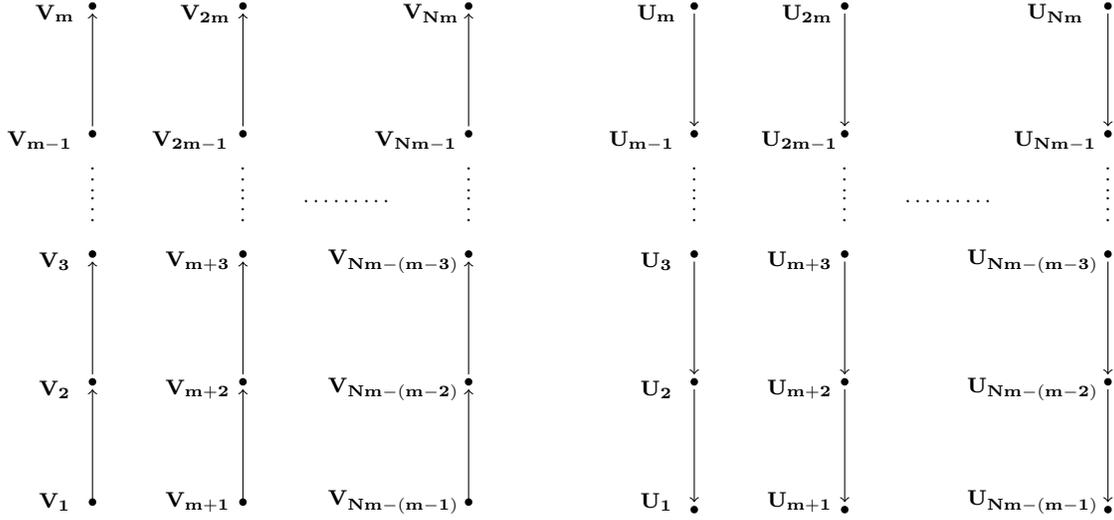
\begin{figure}[ht!]
	\tiny
	\tikzstyle{ver}=[]
	\tikzstyle{vert}=[circle, draw, fill=black!100, inner sep=0pt, minimum width=4pt]
	\tikzstyle{vertex}=[circle, draw, fill=black!00, inner sep=0pt, minimum width=4pt]
	\tikzstyle{edge} = [draw,thick,-]
	%\tikzstyle{edge_style} = [draw=black, line width=2, ultra thick]
	\tikzstyle{node_style} = [circle,draw=blue,fill=blue!20!,font=\sffamily\Large\bfseries]
	\centering
	%\tikzset{->,>=stealth', auto,node distance=1cm,
	%thick,main node/.style={circle,draw,font=\sffamily\Large\bfseries}}
	\tikzset{->-/.style={decoration={
				markings,
				mark=at position #1 with {\arrow{>}}},postaction={decorate}}}
	\begin{tikzpicture}[scale=1]
	\tikzstyle{edge_style} = [draw=black, line width=2mm, ]
	\tikzstyle{node_style} = [draw=blue,fill=blue!00!,font=\sffamily\Large\bfseries]
	%%\node[] (c) at  (3,0) {$\bf{1}$};
	%\node[] (d) at  (6,0) {$\bf{2}$};
	%\node[] (e) at  (9,0) {$\bf{3}$};
	%\node[] (f) at  (12,0) {$\bf{4}$};
	%\node[] (g) at  (15,0) {$\bf{5}$};
	%\node[] (h) at  (18,0) {$\bf{6}$};
	\fill[black!100!](0,0) circle (.05);
	\fill[black!100!](0,1.6) circle (.05);
	\fill[black!100!](0,3.3) circle (.05);
	\fill[black!100!](0, 4.9) circle (.05);
	%\fill[black!100!](3, -4.5) circle (.05);
	\fill[black!100!](0,6.6) circle (.05);
	\draw[->](0,0)--(0,1.5);
	\draw[->](0,1.7)--(0,3.2);
	%\draw [ dashed,->  ] (0,3.4) -- (0,4.8);
	\node[] (g) at  (0,4){$\vdots$};
	\node[] (g) at  (0,4.4){$\vdots$};
	\draw[->](0,5)--(0,6.5);
	\node[] (d) at  (-.5,0) {$\bf{V_1}$};
	\node[] (e) at  (-.5,1.5) {$\bf{V_2}$};
	\node[] (f) at  (-.5,3.2) {$\bf{V_3}$};
	%\draw [ dashed,  ] (-.5,3.4) -- (-.5,4.8);
	\node[] (g) at  (-.7,4.8) {$\bf{V_{m-1}}$};
	\node[] (h) at  (-.5,6.5) {$\bf{V_m}$};
	\fill[black!100!](2,0) circle (.05);
	\fill[black!100!](2,1.6) circle (.05);
	\fill[black!100!](2,3.3) circle (.05);
	\fill[black!100!](2, 4.9) circle (.05);
	\fill[black!100!](2,6.6) circle (.05);
	\draw[->](2,0)--(2,1.5);
	\draw[->](2,1.7)--(2,3.2);
	%\draw [dashed, ->] (2,3.4) -- (2,4.8);
	\node[] (g) at  (2,4){$\vdots$};
	\node[] (g) at  (2,4.4){$\vdots$};
	\draw[->](2,5)--(2,6.5);
	%
	%node
	\node[] (d) at  (1.4,0) {$\bf{V_{m+1}}$};
	\node[] (e) at  (1.4,1.5) {$\bf{V_{m+2}}$};
	\node[] (f) at  (1.4,3.2) {$\bf{V_{m+3}}$};
	%\draw [ dashed,  ] (-.5,3.4) -- (-.5,4.8);
	\node[] (g) at  (1.3,4.8) {$\bf{V_{2m-1}}$};
	\node[] (h) at  (1.5,6.5) {$\bf{V_{2m}}$};
	%\draw [dashed] (2.2,4) -- (4.8,4);
	\node[] (g) at  (3,4){$\cdots$};
	\node[] (g) at  (3.4,4){$\cdots$};
	\node[] (g) at  (3.8,4){$\cdots$};
	\fill[black!100!](5,0) circle (.05);
	\fill[black!100!](5,1.6) circle (.05);
	\fill[black!100!](5,3.3) circle (.05);
	\fill[black!100!](5, 4.9) circle (.05);
	\fill[black!100!](5,6.6) circle (.05);
	\draw[->](5,0)--(5,1.5);
	\draw[->](5,1.7)--(5,3.2);
	%\draw [dashed, ->] (5,3.4) -- (5,4.8);
	\node[] (g) at  (5,4){$\vdots$};
	\node[] (g) at  (5,4.4){$\vdots$};
	\draw[->](5,5)--(5,6.5);
	%node
	\node[] (d) at  (4,0) {$\bf{V_{Nm-(m-1)}}$};
	\node[] (e) at  (4,1.5) {$\bf{V_{Nm-(m-2)}}$};
	\node[] (f) at  (4,3.2) {$\bf{V_{Nm-(m-3)}}$};
	%\draw [ dashed,  ] (-.5,3.4) -- (-.5,4.8);
	\node[] (g) at  (4.3,4.8) {$\bf{V_{Nm-1}}$};
	\node[] (h) at  (4.5,6.5) {$\bf{V_{Nm}}$};
	%
	%%\node[] (c) at  (3,0) {$\bf{1}$};
	%\node[] (d) at  (6,0) {$\bf{2}$};
	%\node[] (e) at  (9,0) {$\bf{3}$};
	%\node[] (f) at  (12,0) {$\bf{4}$};
	%\node[] (g) at  (15,0) {$\bf{5}$};
	%\node[] (h) at  (18,0) {$\bf{6}$};
	\fill[black!100!](8,-.1) circle (.05);
	\fill[black!100!](8,1.6) circle (.05);
	\fill[black!100!](8,3.3) circle (.05);
	\fill[black!100!](8, 4.9) circle (.05);
	%\fill[black!100!](3, -4.5) circle (.05);
	\fill[black!100!](8,6.6) circle (.05);
	\draw[<-](8,0)--(8,1.5);
	\draw[<-](8,1.7)--(8,3.2);
	%\draw [ dashed,<-  ] (8,3.4) -- (8,4.8);
	\node[] (g) at  (8,4){$\vdots$};
	\node[] (g) at  (8,4.4){$\vdots$};
	\draw[<-](8,5)--(8,6.5);
	\node[] (d) at  (7.5,0) {$\bf{U_1}$};
	\node[] (e) at  (7.5,1.5) {$\bf{U_2}$};
	\node[] (f) at  (7.5,3.2) {$\bf{U_3}$};
	%\draw [ dashed,  ] (-.5,3.4) -- (-.5,4.8);
	\node[] (g) at  (7.3,4.8) {$\bf{U_{m-1}}$};
	\node[] (h) at  (7.5,6.5) {$\bf{U_m}$};
	\fill[black!100!](10,-.1) circle (.05);
	\fill[black!100!](10,1.6) circle (.05);
	\fill[black!100!](10,3.3) circle (.05);
	\fill[black!100!](10, 4.9) circle (.05);
	\fill[black!100!](10,6.6) circle (.05);
	\draw[<-](10,0)--(10,1.5);
	\draw[<-](10,1.7)--(10,3.2);
%	\draw [dashed, <-] (10,3.4) -- (10,4.8);
\node[] (g) at  (10,4){$\vdots$};
\node[] (g) at  (10,4.4){$\vdots$};
	\draw[<-](10,5)--(10,6.5);
	%
	%node
	\node[] (d) at  (9.4,0) {$\bf{U_{m+1}}$};
	\node[] (e) at  (9.4,1.5) {$\bf{U_{m+2}}$};
	\node[] (f) at  (9.4,3.2) {$\bf{U_{m+3}}$};
	%\draw [ dashed,  ] (-.5,3.4) -- (-.5,4.8);
	\node[] (g) at  (9.4,4.8) {$\bf{U_{2m-1}}$};
	\node[] (h) at  (9.5,6.5) {$\bf{U_{2m}}$};
%	\draw [dashed] (10.6,4) -- (13.2,4);
\node[] (g) at  (11,4){$\cdots$};
\node[] (g) at  (11.4,4){$\cdots$};
\node[] (g) at  (11.8,4){$\cdots$};
	\fill[black!100!](13.5,-.1) circle (.05);
	\fill[black!100!](13.5,1.6) circle (.05);
	\fill[black!100!](13.5,3.3) circle (.05);
	\fill[black!100!](13.5, 4.9) circle (.05);
	\fill[black!100!](13.5,6.6) circle (.05);
	\draw[<-](13.5,0)--(13.5,1.5);
	\draw[<-](13.5,1.7)--(13.5,3.2);
	%\draw [dashed, <-] (13.5,3.4) -- (13.5,4.8);
	\node[] (g) at  (13.5,4){$\vdots$};
	\node[] (g) at  (13.5,4.4){$\vdots$};
	\draw[<-](13.5,5)--(13.5,6.5);
	%node
	\node[] (d) at  (12.5,0) {$\bf{U_{Nm-(m-1)}}$};
	\node[] (e) at  (12.5,1.5) {$\bf{U_{Nm-(m-2)}}$};
	\node[] (f) at  (12.5,3.2) {$\bf{U_{Nm-(m-3)}}$};
	%\draw [ dashed,  ] (-.5,3.4) -- (-.5,4.8);
	\node[] (g) at  (12.8,4.8) {$\bf{U_{Nm-1}}$};
	\node[] (h) at  (12.8,6.5) {$\bf{U_{Nm}}$};
	%weight
	%\node[] (g) at  (6.8,-1) {$\bf{w(U_{j+1}U_{j})=1=w(V_jV_{j+1}), \text{ for } i=1, \cdots, Nm-1.}$};
	\end{tikzpicture}
	\caption{A weighted acyclic digraph with $w(U_{j+1}U_{j})=1=w(V_jV_{j+1}),$ for $i=1, \cdots, Nm-1.$ }
	\label{fig:hookf1}	 
\end{figure}
To give a combinatorial interpretation of the theorem we use Gessel-Lindstr\"om-Viennot lemma. So first we have to construct a weighted acyclic digraph $\Gamma_{N, m}$ whose path matrix is the matrix $A(N,m)$ (defined as \eqref{hookmat:3}). For  each $i\in [m],$ let us define $\tilde{U}_i=\{U_i, U_{m+i}, \cdots, U_{(N-2)m+i}, U_{(N-1)m+i} \}$ and $\tilde{V}_i=\{V_i, V_{m+i}, \cdots, V_{(N-2)m+i}, V_{(N-1)m+i}\}.$ We declare the vertex set $V(\Gamma_{N, m})=\bigcup\limits_{i\in [m]}\left(\tilde{U}_i\cup\tilde{V}_i\right).$ Now in $\Gamma_{N, m},$ we want the vertices $\tilde{U_i}\cup \tilde{V_i}$ to form a weighted complete bipartite digraph $\Gamma_i \cong K_{N, N}$ with bi-partition sets $\tilde{U}_i$ and $ \tilde{V}_i$ for each $i\in [m].$ The direction of each edge $U_kV_{\ell} (k=tm+i, \ell=rm+i, \text{ where } t, \ell \in \{0, \cdots, N-1\})$ is always taken from $U_k$ towards $V_{\ell}$ and $w(U_kV_{\ell})=x^{(t+1,r+1)}_{m-i+1}-x^{(t+1,r+1)}_{m-i+2}.$ Clearly each $\Gamma_i$ is acyclic. Now we finish the construction of $\Gamma_{N, m}$ by adjoining additional weighted directed edges to $\cup_{i=1}^{m}E(\Gamma_i),$ as depicted in Figure \ref{fig:hookf1}	  
Notice that $\Gamma_{N, m}$ is acyclic. Now we call $U=\{U_1, U_2, \cdots, U_{Nm}\}$ to be the initial set of vertices and $V=\{V_1, V_2, \cdots, V_{Nm}\}$ to be the terminal set of vertices of the graph $\Gamma_{N, m}$. Then the path matrix of the graph $\Gamma_{N, m}$ is the  matrix $A(N,m).$  
\begin{figure}[ht!]
	\tiny
	\tikzstyle{ver}=[]
	\tikzstyle{vert}=[circle, draw, fill=black!100, inner sep=0pt, minimum width=4pt]
	\tikzstyle{vertex}=[circle, draw, fill=black!00, inner sep=0pt, minimum width=4pt]
	\tikzstyle{edge} = [draw,thick,-]
	%\tikzstyle{edge_style} = [draw=black, line width=2, ultra thick]
	\tikzstyle{node_style} = [circle,draw=blue,fill=blue!20!,font=\sffamily\Large\bfseries]
	\centering
	\tikzset{->,>=stealth', auto,node distance=1cm,
		thick,main node/.style={circle,draw,font=\sffamily\Large\bfseries}}
	\tikzset{->-/.style={decoration={
				markings,
				mark=at position #1 with {\arrow{>}}},postaction={decorate}}}
	\begin{tikzpicture}[scale=1]
	\tikzstyle{edge_style} = [draw=black, line width=2mm, ]
	\tikzstyle{node_style} = [draw=blue,fill=blue!00!,font=\sffamily\Large\bfseries]
	%%\node[] (c) at  (3,0) {$\bf{1}$};
	%\node[] (d) at  (6,0) {$\bf{2}$};
	%\node[] (e) at  (9,0) {$\bf{3}$};
	%\node[] (f) at  (12,0) {$\bf{4}$};
	%\node[] (g) at  (15,0) {$\bf{5}$};
	%\node[] (h) at  (18,0) {$\bf{6}$};
	\fill[black!100!](6.05,0) circle (.05);
	\fill[black!100!](3,0) circle (.05);
	\fill[black!100!](8,-2) circle (.05);
	\fill[black!100!](5.95, -4.5) circle (.05);
	\fill[black!100!](3, -4.5) circle (.05);
	\fill[black!100!](1,-2.05) circle (.05);
	\draw[](3,0)--(6,0);
	\draw[](8,-2)--(6.1,0);
	\draw[](3,0)--(1,-2);
	\draw[](8,-2)--(6,-4.5);
	\draw[](3,-4.5)--(5.9,-4.5);
	\draw[](3,-4.5)--(1,-2.1);
	\node (B1) at (3,-4.9)   {$\bf{U_1}$};
	\node (B1) at (5.9,-4.9)   {$\bf{V_1}$};
	\node (B1) at (8.4,-2)   {$\bf{U_3}$};
	\node (B1) at (6.5,0)   {$\bf{V_3}$};
	\node (B1) at (2.6,0)   {$\bf{U_5}$};
	\node (B1) at (.7, -2)   {$\bf{V_5}$};
	%
	%Weight
	\node (B1) at (16, 2.8)   {$\bf{w(U_1V_1)=x^{(1,1)}_{2}}$};
	\node (B1) at (16, 2.2)   {$\bf{w(U_3V_1)=x^{(2,1)}_{2}}$};
	\node (B1) at (16, 1.6)   {$\bf{w(U_3V_3)=x^{(2,2)}_{2}}$};
	\node (B1) at (16, 1)   {$\bf{w(U_5V_3)=x^{(3,2)}_{2}}$};
	\node (B1) at (16, .4)   {$\bf{w(U_5V_5)=x^{(3,3)}_{2}}$};
	\node (B1) at (16, -.2)   {$\bf{w(U_1V_5)=x^{(1,3)}_{2}}$};
	\draw[](3,-4.5)--(6.05,-.05);
	\draw[](3,0)--(5.9, -4.45);
	\draw[](8,-2)--(1.05,-2.05);
	\node (B1) at (16, -.8)   {$\bf{w(U_3V_5)=x^{(2,3)}_{2}}$};
	\node (B1) at (16, -1.4)   {$\bf{w(U_1V_3)=x^{(1,2)}_{2}}$};
	\node (B1) at (16, -2)   {$\bf{w(U_5V_1)=x^{(3,1)}_{2}}$};
	\draw[](3,3)--(3,0.05);
	\draw[](6,0)--(5.9,3);
	\draw[](8,1)--(8,-1.95);
	\draw[](1,-2)--(1, .9);
	\draw[](6,-4.5)--(6,-1.25);
	\draw[](3,-1.2)--(3,-4.45);
	\draw[](3,3)--(5.85,3);
	\draw[](8,1)--(5.95,3);
	\draw[](8,1)--(6,-1.2);
	\draw[](3,-1.2)--(5.94,-1.2);
	\draw[](3,-1.2)--(1.05,1);
	\draw[](3,3)--(.95,1);
	\node (B1) at (2.6,-1.2)   {$\bf{U_2}$};
	\node (B1) at (6.4,-1.3)   {$\bf{V_2}$};
	\node (B1) at (8.4,1)   {$\bf{U_4}$};
	\node (B1) at (5.85,3.3)   {$\bf{V_4}$};
	\node (B1) at (2.6,3)   {$\bf{U_6}$};
	\node (B1) at (.6,.95)   {$\bf{V_6}$};
	%weight of upper dissk
	\node (B1) at (12,-.4)
	{$\bf{w(U_6V_6)=x^{(3,3)}_{1}-x^{(3,3)}_{2}}$};
	\node (B1) at (12,-.9)   {$\bf{w(U_4V_4)=x^{(2,2)}_{1}-x^{(2,2)}_{2}}$};
	\node (B1) at (12, -1.4)   {$\bf{w(U_6V_4)=x^{(3,2)}_{1}-x^{(3,2)}_{2}}$};
	%\node (B1) at (12, -2 )   {$\bf{w()=x_{221}-x_{222}}$};
	%\node (B1) at (8.4,1.9)   {$\bf{x_{231}-x_{232}}$};
	%\node (B1) at (4.3,3.3)   {$\bf{x_{321}-x_{322}}$};
	%\node (B1) at (1.3,2.5)   {$\bf{x_{331}-x_{332}}$};
	%\node (B1) at (8.4,2 )   {$\bf{x_{221}-x_{222}}$};
	%
	\fill[black!100!](3,3) circle (.05);
	\fill[black!100!](5.91,3) circle (.05);
	\fill[black!100!](8,1) circle (.05);
	\fill[black!100!](5.96, -1.2) circle (.05);
	\fill[black!100!](3, -1.2) circle (.05);
	\fill[black!100!](.98,.95) circle (.05);
	\draw[](3,3)--(6, -1.15);
	\draw[](3,-1.2)--(5.82,2.95);
	\draw[](8,1)--(1.08,1.02);
	\node (B1) at (12,2.8)   {$\bf{w(U_6V_2)=x^{(3,1)}_{1}-x^{(3,1)}_{2}}$};
	\node (B1) at (12,2.2)   {$\bf{w(U_2V_4)=x^{(1,2)}_{1}-x^{(1,2)}_{2}}$};
	\node (B1) at (12,1.7)   {$\bf{w(U_4V_6)=x^{(2,3)}_{1}-x^{(2,3)}_{2}}$};
	\node (B1) at (12,.6)   {$\bf{w(U_2V_2)=x^{(1,1)}_{1}-x^{(1,1)}_{2}}$};
	\node (B1) at (12,.1)   {$\bf{w(U_2V_6)=x^{(1,3)}_{1}-x^{(1,3)}_{2}}$};
	\node (B1) at (12,1.1)   {$\bf{w(U_4V_2)=x^{(2,1)}_{1}-x^{(2,1)}_{2}}$};
	\node (B1) at (12,-2.8)   {$\bf{w(U_6U_5)}=\bf{w(U_2U_1)}=\bf{w(U_4U_3)=1}$};
	\node (B1) at (12,-3.4)   {$\bf{w(V_1V_2)}=\bf{w(V_3V_4)}=\bf{w(V_5V_6)=1}$};
	\end{tikzpicture}
	\caption{A weighted acyclic digraph $\Gamma_{3, 2}$. The weight of the  edges appear on right side of this figure.}
	\label{fig:hookf2}	
\end{figure}
So, by  Gessel-Lindstr\"om-Viennot lemma, we can write 
\begin{equation}\label{hookdetform1}
{\det}(A(N,m))=\sum\limits_{\mathcal{P}_{\Gamma_{N, m}}\in VD_{\Gamma_{N, m}}}\text{sgn}(\mathcal{P}_{\Gamma_{N, m}})w(\mathcal{P}_{\Gamma_{N, m}}). 
\end{equation}
\begin{figure}[ht!]
	\tiny
	\tikzstyle{ver}=[]
	\tikzstyle{vert}=[circle, draw, fill=black!100, inner sep=0pt, minimum width=4pt]
	\tikzstyle{vertex}=[circle, draw, fill=black!00, inner sep=0pt, minimum width=4pt]
	\tikzstyle{edge} = [draw,thick,-]
	%\tikzstyle{edge_style} = [draw=black, line width=2, ultra thick]
	\tikzstyle{node_style} = [circle,draw=blue,fill=blue!20!,font=\sffamily\Large\bfseries]
	\centering
	\tikzset{->,>=stealth', auto,node distance=1cm,
		thick,main node/.style={circle,draw,font=\sffamily\Large\bfseries}}
	\tikzset{->-/.style={decoration={
				markings,
				mark=at position #1 with {\arrow{>}}},postaction={decorate}}}
	\begin{tikzpicture}[scale=1]
	\tikzstyle{edge_style} = [draw=black, line width=2mm, ]
	\tikzstyle{node_style} = [draw=blue,fill=blue!00!,font=\sffamily\Large\bfseries]
	%%\node[] (c) at  (3,0) {$\bf{1}$};
	%\node[] (d) at  (6,0) {$\bf{2}$};
	%\node[] (e) at  (9,0) {$\bf{3}$};
	%\node[] (f) at  (12,0) {$\bf{4}$};
	%\node[] (g) at  (15,0) {$\bf{5}$};
	%\node[] (h) at  (18,0) {$\bf{6}$};
	\fill[black!100!](6.05,0) circle (.05);
	\fill[black!100!](3,0) circle (.05);
	\fill[black!100!](8,-2) circle (.05);
	\fill[black!100!](5.95, -4.5) circle (.05);
	\fill[black!100!](3, -4.5) circle (.05);
	\fill[black!100!](1,-2.05) circle (.05);
	\draw[](3,0)--(6,0);
	\draw[](8,-2)--(6.1,0);
	\draw[](3,0)--(1,-2);
	\draw[](8,-2)--(6,-4.5);
	\draw[](3,-4.5)--(5.9,-4.5);
	\draw[](3,-4.5)--(1,-2.1);
	\node (B1) at (3,-4.9)   {$\bf{U_1}$};
	\node (B1) at (5.9,-4.9)   {$\bf{V_1}$};
	\node (B1) at (8.4,-2)   {$\bf{U_3}$};
	\node (B1) at (6.5,0)   {$\bf{V_3}$};
	\node (B1) at (2.6,0)   {$\bf{U_5}$};
	\node (B1) at (.7, -2)   {$\bf{V_5}$};
	%
	%Weight
	\node (B1) at (12,0)   {$\bf{w(U_1V_1)=x^{(1,1)}_{2}}$};
	\node (B1) at (12,-.5)   {$\bf{w(U_3V_1)=x^{(2,1)}_{2}}$};
	\node (B1) at (12,-1)   {$\bf{w(U_3V_3)=x^{(2,2)}_{2}}$};
	\node (B1) at (12,-1.5)   {$\bf{w(U_5V_3)=x^{(3,2)}_{2}}$};
	\node (B1) at (12,-2)   {$\bf{w(U_5V_5)=x^{(3,3)}_{2}}$};
	\node (B1) at (12,-2.5)   {$\bf{w(U_1V_5)=x^{(1,3)}_{2}}$};
	\node (B1) at (12,-3)   {$\bf{w(U_3V_5)=x^{(2,3)}_{2}}$};
	\node (B1) at (12,-3.5)   {$\bf{w(U_1V_3)=x^{(1,2)}_{2}}$};
	\node (B1) at (12,-4)   {$\bf{w(U_5V_1)=x^{(3,1)}_{2}}$};
	%\node (B1) at (6.9,-3.9)   {$\bf{x_{212}}$};
	%\node (B1) at (7.4,-.9)   {$\bf{x_{222}}$};
	%\node (B1) at (4.5,.2)   {$\bf{x_{322}}$};
	%\node (B1) at (1.9,-.7)   {$\bf{x_{332}}$};
	%\node (B1) at (2, -3.9)   {$\bf{x_{132}}$};
	%
	\draw[](3,-4.5)--(6.05,-.05);
	\draw[](3,0)--(5.9, -4.45);
	\draw[](8,-2)--(1.05,-2.05);
	%
	%\node (B1) at (1.9,-2.3)   {$\bf{x_{232}}$};
	%	\node (B1) at (3.9,-4)   {$\bf{x_{122}}$};
	%\node (B1) at (5.2,-4)   {$\bf{x_{312}}$};
	%
	%\node (B1) at (13, 1.8)   {$\bf{w(U_1V_1)=x_{112}}$};
	%\node (B1) at (13, 1.2)   {$\bf{w(U_3V_1)=x_{212}}$};
	%\node (B1) at (13, .6)   {$\bf{w(U_3V_3)=x_{222}}$};
	%\node (B1) at (13, 0)   {$\bf{w(U_5V_3)=x_{322}}$};
	%\node (B1) at (13, -.6)   {$\bf{w(U_5V_5)=x_{332}}$};
	%\node (B1) at (13, -.12)   {$\bf{w(U_1V_5)=x_{132}}$};
	%\node (B1) at (13, -.18)   {$\bf{w(U_3V_5)=x_{232}}$};
	%\node (B1) at (13, -.24)   {$\bf{w(U_1V_3)=x_{122}}$};
	%\node (B1) at (13, -.30)   {$\bf{w(U_5V_1)=x_{312}}$};
	\end{tikzpicture}
	\caption{A weighted acyclic digraph $\Gamma_1$. The weight of the corresponding edges appear on figure.}
	\label{fig:hookf3}	
\end{figure}
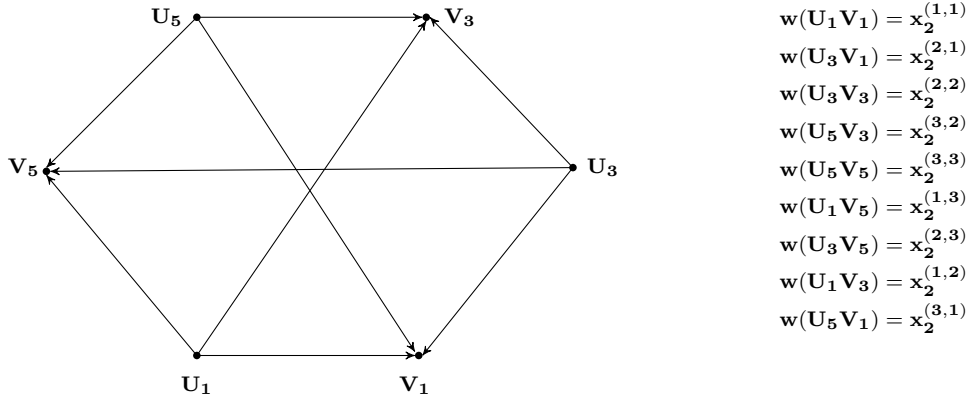
\begin{figure}[ht!]
	\tiny
	\tikzstyle{ver}=[]
	\tikzstyle{vert}=[circle, draw, fill=black!100, inner sep=0pt, minimum width=4pt]
	\tikzstyle{vertex}=[circle, draw, fill=black!00, inner sep=0pt, minimum width=4pt]
	\tikzstyle{edge} = [draw,thick,-]
	%\tikzstyle{edge_style} = [draw=black, line width=2, ultra thick]
	\tikzstyle{node_style} = [circle,draw=blue,fill=blue!20!,font=\sffamily\Large\bfseries]
	\centering
	\tikzset{->,>=stealth', auto,node distance=1cm,
		thick,main node/.style={circle,draw,font=\sffamily\Large\bfseries}}
	\tikzset{->-/.style={decoration={
				markings,
				mark=at position #1 with {\arrow{>}}},postaction={decorate}}}
	\begin{tikzpicture}[scale=1]
	\tikzstyle{edge_style} = [draw=black, line width=2mm, ]
	\tikzstyle{node_style} = [draw=blue,fill=blue!00!,font=\sffamily\Large\bfseries]
	\draw[](3,3)--(5.85,3);
	\draw[](8,1)--(5.95,3);
	\draw[](8,1)--(6,-1.2);
	\draw[](3,-1.2)--(5.94,-1.2);
	\draw[](3,-1.2)--(1.05,1);
	\draw[](3,3)--(.95,1);
	\node (B1) at (2.6,-1.2)   {$\bf{U_2}$};
	\node (B1) at (6.4,-1.3)   {$\bf{V_2}$};
	\node (B1) at (8.4,1)   {$\bf{U_4}$};
	\node (B1) at (5.85,3.3)   {$\bf{V_4}$};
	\node (B1) at (2.6,3)   {$\bf{U_6}$};
	\node (B1) at (.6,.95)   {$\bf{V_6}$};
	%weight of upper dissk
	\node (B1) at (12,3)   {$\bf{w(U_2V_2)=x^{(1,1)}_{1}-x^{(1,1)}_{2}}$};
	\node (B1) at (12,2.5)   {$\bf{w(U_4V_2)=x^{(2,1)}_{1}-x^{(2,1)}_{2}}$};
	\node (B1) at (12,2)   {$\bf{w(U_4V_4)=x^{(2,2)}_{1}-x^{(2,2)}_{2}}$};
	\node (B1) at (12,1.5)   {$\bf{w(U_6V_4)=x^{(3,2)}_{1}-x^{(3,2)}_{2}}$};
	\node (B1) at (12,1)   {$\bf{w(U_6V_6)=x^{(3,3)}_{1}-x^{(3,3)}_{2}}$};
	\node (B1) at (12,.5)   {$\bf{w(U_2V_6)=x^{(1,3)}_{1}-x^{(1,3)}_{2}}$};
	\node (B1) at (12,0)   {$\bf{w(U_2V_4)=x^{(1,2)}_{1}-x^{(1,2)}_{2}}$};
	\node (B1) at (12,-.5)   {$\bf{w(U_6V_2)=x^{(3,1)}_{1}-x^{(3,1)}_{2}}$};
	\node (B1) at (12,-1)   {$\bf{w(U_4V_6)=x^{(2,3)}_{1}-x^{(2,3)}_{2}}$};
	%\node (B1) at (12,-3.5)   {$\bf{w(U_1V_1)=x_{112}}$};
	%\node (B1) at (4.5,-1.6)   {$\bf{x_{111}-x_{112}}$};
	%\node (B1) at (8,-.5)    {$\bf{x_{211}-x_{212}}$};
	%\node (B1) at (8.4,1.9)   {$\bf{x_{231}-x_{232}}$};
	%\node (B1) at (4.3,3.3)   {$\bf{x_{321}-x_{322}}$};
	%\node (B1) at (1.3,2.5)   {$\bf{x_{321}-x_{322}}$};
	%\node (B1) at (.8,-.2)   {$\bf{x_{331}-x_{332}}$};
	%\node (B1) at (8.4,2 )   {$\bf{x_{221}-x_{222}}$};
	%
	\fill[black!100!](3,3) circle (.05);
	\fill[black!100!](5.91,3) circle (.05);
	\fill[black!100!](8,1) circle (.05);
	\fill[black!100!](5.96, -1.2) circle (.05);
	\fill[black!100!](3, -1.2) circle (.05);
	\fill[black!100!](.98,.95) circle (.05);
	\draw[](3,3)--(6, -1.15);
	\draw[](3,-1.2)--(5.82,2.95);
	\draw[](8,1)--(1.08,1.02);
	%
	%\node (B1) at (3,1.8)   {$\bf{x_{311}-x_{312}}$};
	%\node (B1) at (6,1.8)   {$\bf{x_{121}-x_{122}}$};
	%\node (B1) at (3,.8)   {$\bf{x_{231}-x_{232}}$};
	\end{tikzpicture}
	\caption{A weighted acyclic digraph $\Gamma_2$. The weight of the corresponding edges appear on figure.}
	\label{fig:hookf4}	
\end{figure}
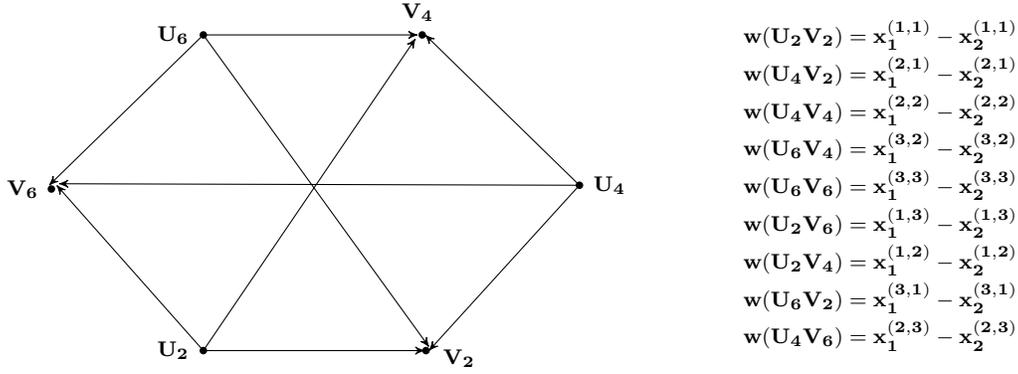 
Now we have to characterize all the vertex disjoint path systems in the graph $\Gamma_{N, m}.$ The following lemmas give the complete characterization of all vertex disjoint path systems in $\Gamma_{N, m}.$
\begin{lemma}\label{combi lemma1}
Let $\mathcal{P}_{\Gamma_{N, m}}\in VD_{\Gamma_{N, m}}$ and $P$ be a path in $\mathcal{P}_{\Gamma_{N, m}}.$  Then $\ell(P)=1.$
\end{lemma} 
\begin{proof}
Let $P$ be a path in  $\mathcal{P}_{\Gamma_{N, m}}$ from the initial vertex $U_{l}$ to the terminal vertex $V_{k},$ for some $U_{l}\in U$ and $V_{k}\in V.$  Then we show that the length of the path $P$ is $1.$ In fact, length of $P>1$ implies, $P$ contains at least  three distinct vertices. Then $P$ is of the form either $U_lU_i\cdots V_k$ or $U_lV_j\cdots V_k,$ for some $U_i\in U$ and $V_j\in V,$ i.e, the second vertex of $P$ (a vertex next to the initial vertex in $P$) is either $U_i$ or $V_j.$  Suppose $P$ is $U_lU_i\cdots V_k.$ Since $\mathcal{P}_{\Gamma_{N, m}}$ is a path system from $U$ to $V,$ we must have another path $\acute{P}$ in $\mathcal{P}_{\Gamma_{N, m}},$ whose initial vertex is $U_i.$ This contradicts that the path system $\mathcal{P}_{\Gamma_{N, m}}$ is vertex disjoint. Again if $P$ is $U_lV_j\cdots V_k,$ then in $\mathcal{P}_{\Gamma_{N, m}}$ we get another path $\tilde{P}$ whose terminal vertex is $V_j,$ which also contradicts the fact $\mathcal{P}_{\Gamma_{N, m}}$ is vertex disjoint. Therefore, the length of each path in any vertex disjoint path system must be $1.$
\end{proof}
In this portion we will think about the vertex disjoint path systems in  subgraphs $\Gamma_1, \Gamma_2, \cdots, \Gamma_m$ of the graph $\Gamma_{N, m}.$ For each $\Gamma_i  (i=1, \cdots, m)$ we choose $\tilde{U}_i, \tilde{V}_i$ as the initial and terminal set of vertices respectively. It can be shown that the path matrix of the graph $\Gamma_i$ is the matrix $X_{(m-i+1, m-i+2)},$ defined as \eqref{hookmat:1}. From our construction of the graph $\Gamma_{N. m},$ it is evident that \[VD_{\Gamma_{N, m}}=\{(\mathcal{P}, \sigma)= (\mathcal{P}^1\cup \mathcal{P}^2\cup\cdots\cup \mathcal{P}^m, \sigma_1\sigma_2\cdots\sigma_m), \text{ where } (\mathcal{P}^i, \sigma_i) \text{ is a path system of } \Gamma_i\}.\]    
%\begin{lemma}\label{combi lemma 2}
%Let $\mathcal{P}_{\Gamma_i}\in VD_{\Gamma_i},$ for $i=1, \cdots, m.$ Then $\bigoplus\limits _{i=1}^{m}\mathcal{P}_{\Gamma_i}\in VD_{\Gamma_{N, m}}.$
%\end{lemma}
%%\begin{proof}
%The proof is clear from the definition of $\bigoplus\limits _{i=1}^{m}\mathcal{P}_{\Gamma_i}.$
%\end{proof}
%\begin{lemma}\label{combi lemma 3}
%Let $\mathcal{P}_{\Gamma_{N, m}}\in VD_{\Gamma_{N, m}}.$ Then $\mathcal{P}_{\Gamma_{N, m}}=\bigoplus\limits _{i=1}^{m}\mathcal{P}_{\Gamma_i},$ where $\mathcal{P}_{\Gamma_i}\in VD_{\Gamma_i}.$ 
%\end{lemma}
%\begin{proof}
%Notice that in $\Gamma_{N, m}$ each vertical edge (edge joining the vertices of $\Gamma_i$ and $\Gamma_{i+1}$ for $i=1, \cdots, m-1$) has end vertices either in $U$ or in $V$ but not in both. So, no vertical edge is a member of vertex disjoint path system. Also by Lemma \ref{combi lemma1}, each path of any vertex disjoint path system of $\Gamma_{N, m}$ is of length $1.$ Consequently $\mathcal{P}_{\Gamma_{N, m}}=\bigoplus\limits _{i=1}^{m}\mathcal{P}_{\Gamma_i}.$ 
%\end{proof}
Therefore, we can write 
\begin{align*}\label{hookdetform2}
\sum\limits_{(\mathcal{P}, \sigma)\in VD_{\Gamma_{N, m}}}&\text{sgn}(\mathcal{P}, \sigma)w(\mathcal{P}, \sigma)=\\ &\left(\sum\limits_{(\mathcal{P}^1, \sigma_1)\in VD_{\Gamma_1}}\text{sgn}(\mathcal{P}^1, \sigma_1)w(\mathcal{P}^1, \sigma_1) \right)\cdots \left(\sum\limits_{(\mathcal{P}^m, \sigma_m)\in VD_{\Gamma_m}}\text{sgn}(\mathcal{P}^m, \sigma_m)w(\mathcal{P}^m, \sigma_m)\right).
\end{align*}
Now, $\left(\sum\limits_{(\mathcal{P}^i, \sigma_i)\in VD_{\Gamma_i}}\text{sgn}(\mathcal{P}^i, \sigma_i)w(\mathcal{P}^i, \sigma_i)\right)=\det(X_{(m-i+1, m-i+2)}),$  where $i\in[m]$ and $X_{(m-i+1, m-i+2)}$ is the path matrix associated to the graph $\Gamma_i.$ Hence the theorem.
\end{proof}
\begin{example}
If we take $N=3, m=2,$ then the weighted acyclic digraph $\Gamma_{3, 2}$ is Figure \ref{fig:hookf2} and $\Gamma_1, \Gamma_2$ are Figures \ref{fig:hookf3}, \ref{fig:hookf4} respectively.  Moreover the path matrix associated to $\Gamma_{3, 2}$ is $A(3,2)$ defined as \eqref{A_{3,2}hookmat:5}. Consider a vertex disjoint path system  $U_1V_1, U_2V_2, \cdots, U_6V_6$ in $\Gamma,$  where each path $U_iV_i (i=1, \cdots, 6)$ is an edge.  Notice that, three paths $U_1V_1, U_3V_3, U_5V_5$ from the above path system is a vertex disjoint path system in $\Gamma_1$ and three  another paths  $U_2V_2, U_4V_4, U_6V_6$ is a vertex disjoint path system in $\Gamma_2.$  Again $U_1V_5, U_2V_2, U_3V_1, U_4V_6, U_5V_3,  U_6V_4$ is vertex disjoint path system in $\Gamma$. For this vertex disjoint path system, $U_1V_5, U_5V_3, U_3V_1$ is a vertex disjoint path system in $\Gamma_1$ and $U_2V_2, U_4V_6, U_6V_4$ is a vertex disjoint path system in $\Gamma_2.$ 	
\end{example}
\begin{rem}
In exactly the same way we can give combinatorial explanation of all other theorems. In those cases also we use $\Gamma_{N, m}$ as combinatorial object possibly permuting some of its vertices.
\end{rem}
\subsection*{Acknowledgment} I would like to thank my mentor Prof. Arvind Ayyer for giving this problem, insightful discussion to solve the problem and valuable suggestions in the preparation of this paper. The author was supported by Department of Science and Technology grant EMR/2016/006624 and partly supported by  UGC Centre for Advanced Studies. Also the author was supported by NBHM Post Doctoral Fellowship grant 0204/52/2019/RD-II/339.    
\bibliographystyle{amsplain}
\bibliography{gen-inv-lcp}

\providecommand{\bysame}{\leavevmode\hbox to3em{\hrulefill}\thinspace}
\providecommand{\MR}{\relax\ifhmode\unskip\space\fi MR }
\providecommand{\MRhref}[2]{%
  \href{http://www.ams.org/mathscinet-getitem?mr=#1}{#2}
}
\providecommand{\href}[2]{#2}
\begin{thebibliography}{10}

\bibitem{15}
M.~Aigner, \emph{A course in enumeration}, Graduate Texts in Mathematics, vol.
  238, Springer, Berlin, 2007.

\bibitem{18}
A.~Ayyer, \emph{Determinants and perfect matchings}, J. Combin. Theory Ser. A
  \textbf{120} (2013), no.~1, 304--314.

\bibitem{14}
S.~Bera and S.~K. Mukherjee, \emph{Combinatorial proofs of some determinantal
  identities}, Linear Multilinear Algebra \textbf{66} (2018), no.~8,
  1659--1667.

\bibitem{27}
G.~Bhatnagar and C.~Krattenthaler, \emph{Spiral determinants}, Linear Algebra
  Appl. \textbf{529} (2017), 374--390.

\bibitem{41}
I.~Gessel and G.~Viennot, \emph{Binomial determinants, paths, and hook length
  formulae}, Adv. in Math. \textbf{58} (1985), no.~3, 300--321.

\bibitem{26}
C.~Krattenthaler, \emph{Advanced determinant calculus}, vol.~42, 1999, The
  Andrews Festschrift (Maratea, 1998), pp.~Art. B42q, 67.

\bibitem{25}
\bysame, \emph{Advanced determinant calculus: a complement}, Linear Algebra
  Appl. \textbf{411} (2005), 68--166.

\bibitem{32}
L.~G. Molinari, \emph{Determinants of block tridiagonal matrices}, Linear
  Algebra Appl. \textbf{429} (2008), no.~8-9, 2221--2226.

\bibitem{40}
S.~K. Mukherjee and S.~Bera, \emph{Combinatorial proofs of the
  {N}ewton-{G}irard and {C}hapman--{C}ostas-{S}antos identities}, Discrete
  Math. \textbf{342} (2019), no.~6, 1577--1580.

\bibitem{34}
E.~M. Murman and S.~S. Abarbanel, \emph{Progress and supercomputing in
  computational fluid dynamics}, Proceedings of U. S.-Israel Workshop,
  Birkh\"auser Boston, Inc., Germany,, 1984.

\bibitem{33}
O.~Popescu, C.~Rose, and D.~C. Popescu, \emph{Maximizing the determinant for a
  special class of block-partitioned matrices}, Math. Probl. Eng. (2004),
  no.~1, 49--61.

\bibitem{39}
P.~D. Powell, \emph{Calculatting determinants of block matrices},
  arXiV:1112.4379v1 [math.RA] (2011).

\bibitem{29}
C.~Roessner, S.~Ratti and W.~Weise, \emph{Polyakov loop, diquarks and the
  two-flavour phase diagram}, Phys. Rev. D. \textbf{75} (2007), 034007--034004.

\bibitem{30}
B.~Sasaki, C.~Friman and K.~Redlich, \emph{Susceptibilities and the phase
  structure of a chiral model with polyakov loops}, Phys. Rev. D. \textbf{75}
  (2007), 074013--074029.

\bibitem{11}
D.~Zeilberger, \emph{A combinatorial proof of newton's identity}, Discrete
  Mathematics \textbf{49} (1984), 319.

\bibitem{36}
D.~Zheng, \emph{Matrix methods for determinants of {P}ascal-like matrices},
  Linear Algebra Appl. \textbf{577} (2019), 94--113.

\end{thebibliography}
\end{document}